\PassOptionsToPackage{table}{xcolor}
\documentclass[12pt]{amsart}
\usepackage[margin=2cm]{geometry}
\usepackage[
backend=biber,
style=alphabetic,
sorting=nyt,
giveninits=true
]{biblatex}

\usepackage{graphicx, array, url, tcolorbox, xcolor, hyperref, amsmath, amsfonts, amsthm, amssymb, amstext}
\usepackage{tikz}
\usepackage{float}
\usepackage{caption}

\usepackage{cleveref}
\usepackage{wrapfig} 
\usepackage{graphicx, latexsym, amsthm,alltt,color, listings, enumerate,multicol, parskip,vwcol, tikz,pgfplots,enumitem, pdfsync}
\usetikzlibrary{positioning}
\usetikzlibrary{calc}
\usetikzlibrary{arrows,chains,scopes}
\usetikzlibrary{cd}
\usetikzlibrary{decorations.markings}
\tcbuselibrary{theorems}
\usetikzlibrary{arrows.meta}
\usetikzlibrary{trees}
\usetikzlibrary{arrows,chains,scopes}
\usetikzlibrary{cd}
\usetikzlibrary{decorations.markings}

\usepackage[utf8]{inputenc}
\usepackage{xcolor}
\usepackage{xspace}
\usepackage{ulem}   
\usepackage{algorithm2e}
\usepackage{enumitem}
\newlist{tconds}{enumerate}{1}
\setlist[tconds]{label=\(\mathsf{(T_{\arabic*})}\), ref=\(\mathsf{T_{\arabic*}}\), leftmargin=*, labelsep=.6em}

\newtheorem{theorem}{Theorem}
\theoremstyle{example}
\newtheorem{example}[theorem]{Example}
\newtheorem*{example*}{Example}
\newtheorem{definition}[theorem]{Definition}
\newtheorem*{theorem*}{Theorem}
\newtheorem{remark}[theorem]{Remark}
\newtheorem{corollary}[theorem]{Corollary}
\newtheorem{lemma}[theorem]{Lemma}
\newtheorem{proposition}[theorem]{Proposition}

\theoremstyle{remark}
\newtheorem{convention}[theorem]{Convention}
\newcommand{\girth}{\mathrm{girth}}
\newcommand{\hgirth}{\mathrm{half\text{-}girth}}

\newcommand{\va}{\mathcal{V}_A}
\newcommand{\vb}{\mathcal{V}_B}
\newcommand{\ecal}{\mathcal{E}}

\definecolor{royalblue}{HTML}{0071BC}
\definecolor{springgreen}{HTML}{C6DC67}
\definecolor{lavender}{HTML}{AF72B0}
\definecolor{arsenic}{rgb}{0.23, 0.27, 0.29}

\newcommand{\magenta}[1]{\textcolor{magenta}{#1}}

\title{On zero-divisors and units in group rings of torsion-free CAT$(0)$ groups}

\author[First Author]{Manisha Garg, Igor Mineyev}

\address{University of Illinois, Urbana-Champaign, Department of Mathematics, 1409 West Green Street, Urbana, IL 61801}
\email{\{manisha8, mineyev\}@illinois.edu}

\subjclass[2020]{Primary 20F65, 20F67, 20C07, 20F05, 05C25, 57M05, 20-04, 20-08; Secondary 16U60, 05C57, 51-08, 51F99, 52-08, 55-08, 57M07, 57M15, 57M60, 57Z25,%
 05C85, 05C15, 20F06, 90C35, 52B05, 52C99, 57-08, 57K20, 57M10, 57Q05, 16S34,%
 68R10, 68R99, 68V99, 05C20, 05C38, 05C90}

\keywords{Kaplansky conjectures, group ring, group algebra, unit, zero-divisor, graph, product
structure, taiko, cell complex, directed acyclic graph}
\begin{document}

\maketitle
\begin{abstract}
   This paper addresses two of Kaplansky's conjectures concerning group rings $K[G]$, where $K$ is a field and $G$ is a torsion-free group: the zero-divisor conjecture, which asserts that $K[G]$ has no non-trivial zero-divisors, and the unit conjecture, which asserts that $K[G]$ has no non-trivial units. While the zero-divisor conjecture still remains open, the unit conjecture was disproven by Gardam in 2021. The search for more counterexamples remains an open problem.

     Let $m$ and $n$ be the cardinality of support of two non-trivial elements $\alpha, \beta \in \mathbb{F}_2[G]$, respectively. We address these conjectures by introducing a process called \text{left alignment} and  recursively constructing the taikos of size $(m,n)$ which would yield counterexamples to both conjectures over the field $\mathbb{F}_2$ if they satisfy conditions $\mathsf{T}_1-\mathsf{T}_4$ given in \cite{Mineyev2024}. We also present a computer-search method that can be utilized to search for counterexamples of a certain geometry by significantly pruning the search space. We prove that a class of CAT(0) groups with certain geometry cannot be counterexamples to these conjectures. Moreover, we prove that for $ 1\le m \le 5$ and $n$ any positive integer, there are no counterexamples to the conjectures such that the associated oriented product structures are of type $(m,n)$. With the aid of computer, we prove that, in fact, there are no such counterexamples of the length combination $(m,n)$ where $1\le m \le 13$ and $1\le n \le 13$. 
    
\end{abstract}

\begingroup
  \setlength{\parskip}{0pt}  
  \tableofcontents
\endgroup
\section{Introduction}

In 1956, I. Kaplansky \cite{Kaplansky1956, Kaplansky1970} presented a list of several algebraic conjectures concerning group rings $K[G]$, where $K$ is a field and $G$ is a torsion-free group. Our paper focuses on two of these conjectures:
\begin{enumerate} \item The \textit{zero-divisor Conjecture}: $K[G]$ has no non-trivial zero-divisors. \item The \textit{unit conjecture}: $K[G]$ has no non-trivial units. \end{enumerate}

The unit conjecture has been studied extensively, with its origins traced back to Higman \cite{Higman1940, Higmanthesis}, who, in his doctoral thesis among other results, proved that all locally indicable groups satisfy this conjecture. In 2021, Gardam \cite{Gardam2021} discovered the first counterexample to the unit conjecture over the field $\mathbb{F}_2$, which was subsequently extended to fields of prime characteristic by Murray \cite{Murray2021}.
 
The unit conjecture implies the zero-divisor conjecture (see \cite{Passman1977}, Lemma 13.1.2) and both conjectures are known to hold true for a large class of groups. For torsion-free supersolvable groups \cite{Formanek1973} and polycyclic-by-finite groups \cite{FarkasSnider1976, Brown1976}, the zero-divisor conjecture is true. Both conjectures hold true for the class of unique-product groups \cite{Cohen1974}. Examples of unique-product groups include linearly ordered groups and torsion-free abelian groups. It might be fruitful to look for counterexamples among torsion-free non-unique product groups.

Constructing torsion-free non-unique product groups has been a challenging task. In 1987, Rips and Segev \cite{RipsSegev1987} constructed the first torsion-free non-unique product group using small cancellation theory. A year later, Promislow \cite{Promislow1988}, gave a simpler example of a non-unique product group $P$ given as
\begin{equation}\label{eqn:promislow_group}	
	P = \langle x, y \mid (x^2)^y = x^{-2}, (y^2)^x = y^{-2} \rangle
\end{equation}

This group is known by several names in literature including Promislow group, Passman and Hantzsche–Wendt group. It is supersolvable and in particular, it is polycyclic. $P$ contains the subgroup $H = \langle x^2, y^2, [x, y] \rangle \cong \mathbb{Z}^3$, extended by the Klein Four group $V_4 = C_2 \times C_2 \cong \langle xH, yH \rangle$. Apart from \cite{Steenbock2015, ArzhantsevaSteenbock2023, GruberMartinSteenbock2015, Soelberg2018} very few examples of non-unique product groups are known. Gardam's counterexample to the unit conjecture comes from Promislow group $P$ given in \eqref{eqn:promislow_group}. However, $P$ is known to satisfy the zero-divisor conjecture, since it is supersoluble (see \cite{CravenPappas2013}). Moreover, $P$ is a three-dimensional CAT(0) group. Gardam’s counterexample thus motivates the search for further counterexamples within the class of CAT(0) groups.

Several combinatorial and computational approaches have been utilized to search for counterexamples of these longstanding conjectures (for instance, see \cite{Schweitzer2013, Soelberg2018, Carter2014, BondarenkoJuschenko2024, AbdollahiTaheri2018, AbdollahiTaheri2019}). The second author in Theorem \ref{thm:27}, \cite{Mineyev2024} presents a list of sufficient combinatorial conditions that imply counterexamples of these conjectures.

In this paper, we address the zero-divisor conjecture and the unit conjecture by introducing a process called `\textit{left alignment}' and constructing an algorithm based on four combinatorial conditions $\mathsf{T}_1$ through $\mathsf{T}_4$ to search for counterexamples in the group ring $K[G]$, where $K = \mathbb{F}_2$ and $G$ is a CAT(0) group constructed in \cite{Mineyev2024}.

Let $\alpha, \beta \in \mathbb{F}_2[G]$ be elements expressed as $\alpha = a_1 + \cdots + a_m$ and $\beta = b_1 + \cdots + b_n$. Recall that the support of an element $\alpha \in \mathbb{F}_2[G]$, denoted as $\text{supp}(\alpha)$, is the set $\{a_1, \ldots, a_m\}$. Let $A:= \text{supp}(\alpha)$ and $B:=\text{supp}(\beta)$. Since the field is $\mathbb{F}_2$, we can assume that all $a_i$ are distinct and all $b_j$ are distinct and thus $|A|=m$ and $|B|=n$. The product structure associated with $\alpha \beta$ is a combinatorial construct defined in terms of a partition $P$ of $A \times B$. This product structure, denoted by $\Pi = (A, B, P)$, is said to be ``oriented" if it satisfies certain conditions described in Section 2. By analyzing these oriented product structures, we aim to identify cases where $\Pi$ provides a counterexample to either the zero-divisor conjecture or the unit conjecture. The sufficient conditions provided in Theorem \ref{thm:27} relate the structure of $\Pi$ to properties such as girth and non-degeneracy of the associated combinatorial object. 
In this paper, we prove that one of the conditions never holds for an oriented product structure:
\begin{theorem}\label{thm:main}
    Let $A$ and $B$ be two finite sets and $\Pi = (A, B, P)$ be an orientable even or odd product structure with associated middle-link graph $\mathsf{L}_1$. Then the half-girth of $\mathsf{L}_1$ is at most 4.
\end{theorem}

A counterexample is said to be of type $(m, n)$ if there exists an even or odd product structure $\Pi = (A, B, P)$ where $|A| = m$, $|B| = n$, and the partition $P$ satisfies conditions $\mathsf{T}_1$ through $\mathsf{T}_4$. When $mn$ is even, such a $\Pi$ yields a counterexample to the zero-divisor conjecture. When $mn$ is odd, $\Pi$ serves as a counterexample to the unit conjecture. We also prove:
\begin{theorem}\label{thm:noexampletypemn}
    There are no counterexamples of type $(m,n)$ or $(n,m)$ where $m\le 5$ and $n$ is a positive integer.
\end{theorem}
For larger values of $m$ and $n$, the combinatorial arguments to prove the existence or non-existence of counterexamples of type $(m,n)$ become tedious as the number of cases grow exponentially with size. Thus it is beneficial to employ a computer-search. We introduce a process called left-alignment which is similar in spirit to the adequacy condition employed in \cite{Schweitzer2013}. We develop an algorithm for a fixed length combination $(m,n)$ and utilize the left-alignment condition and the combinatorial conditions $\mathsf{T}_1 - \mathsf{T}_4$ to prune the search space for the counterexamples significantly. We obtain the following strengthened variant of Theorem \ref{thm:noexampletypemn}.
\begin{theorem}[Computer-assisted]\label{thm:computer_assisted}
    There are no counterexamples of type $(m,n)$ or $(n,m)$ where 
    \begin{enumerate}
        \item $1\le m\le 13$ and $1\le n\le 13$, or, 
        \item $m\in \{6,7\}$ and $1\le n \le 200$
    \end{enumerate}
    
\end{theorem}

Since these conjectures are known to hold for a large class of groups, a new direction towards identifying counterexamples was inspired by Gardam's example, which involves elements with a support length of 21. Extensive work has been done to rule out cases with smaller support lengths; see, for instance, \cite{DykemaHeisterJuschenko15}. However, since the number of potential cases grows exponentially with the product $mn$, historically, we have been able to reach until the sizes (5,n). 

In this paper, we effectively prune a significant number of pairs, narrowing the search space further to look for counterexamples. The Algorithm \ref{alg:dfs-search-dag} is scalable to incorporate any topological or geometric condition that can be realized combinatorially. Furthermore, it can be generalized to facilitate searches in higher-dimensional CAT(0) cell complexes which might be the right space to look for counterexamples. 
 
\subsection*{Structure of the paper} 
Section 2 provides the necessary definitions and preliminaries, including the notion of oriented product structures, two graphs, \textit{taiko} and the \textit{middle-link}, associated to product structures and the combinatorial conditions $\mathsf{T}_1$ through $\mathsf{T}_4$. Section 3 presents the process of left-alignment for building subpartitions of $A\times B$ and an algorithm that can be utilized to create a list of all product substructures up to isomorphism. Section 4 contains the proof of Theorem \ref{thm:main} and Theorem \ref{thm:noexampletypemn}. In Section 5 we state computational results and Theorem 3 and present several partial results regarding existing of product structures satisfying a subset of conditions $\mathsf{T}_1-\mathsf{T}_4$.

\subsection*{Acknowledgments} This work was partially supported by Campus Research Board award RB22075. The first author acknowledges the support provided by the Margaret McNamara Education Grant. The authors also thank Haizi Yu for helpful conversations about the algorithm. The first author thanks Nathan Dunfield and Jake Rasmussen for many interesting conversations, and her advisor, Jeremy Tyson, for his guidance throughout this project.

\section{Preliminaries}
      In this paper, we will consider the case where the field is $\mathbb{F}_2$. Let $\alpha, \beta \in \mathbb{F}_2[G]$ such that $\alpha = a_1 + \cdots + a_m$ and $\beta = b_1 + \cdots + b_n$. Recall that for an element $\alpha = a_1 + \cdots + a_m$, the support of $\alpha$, denoted $\text{supp}(\alpha)$, is the set $\{a_1, \ldots, a_m\}$. Since the field is $\mathbb{F}_2$, we may assume that $|\text{supp}(\alpha)| = m$ and $|\text{supp}(\beta)| = n$.  
    
    Now, let $\alpha \beta = 0$ and $mn$ be even. This implies  
    \begin{equation}\label{eqn:productab}
        \sum_{i=1}^{m} \sum_{j=1}^{n} a_i b_j = 0.
    \end{equation}

    Since all terms in $\alpha \beta$ cannot be distinct, at least two terms must be identical. Without loss of generality, assume $a_1 b_1 = a_2 b_2$. This implies  $a_1 b_1 + a_2 b_2 = 0$. By induction on the number of terms in the sum, it follows that all terms in $\alpha\beta$ must occur in pairs of identical elements. 

    As a consequence, the product $\alpha \beta$ can be represented by a complete bipartite graph $\mathcal{G}(A, B)$, where $A =\{a_1, \ldots, a_m\}$ and $B=\{b_1, \ldots, b_n\}$ are two disjoint vertex sets. The edge set of $\mathcal{G}(A, B)$ is given by $\{(a, b) \mid a \in A, b \in B\}$, which corresponds to the Cartesian product $A \times B$. Each term in the product $\alpha\beta$ \eqref{eqn:productab} corresponds to an edge in $A\times B$. A pair of two identical terms corresponds to the pairing of two vertical edges. Note that if $a_ib_j = a_{i'}b_{j'}$ then $a_i \neq a_{i'}$ and $b_j \neq b_{j'}$, because otherwise either $|\text{supp}(\alpha)| < m$ or $|\text{supp}(\beta)| < n$. Thus, in a pair, no two edges share the same vertex. We will call each pair a cell.

    Similarly, for $mn$ odd, consider the equation $\alpha\beta=1$. This would lead to pairing of $mn-1$ edges in the complete bipartite graph $\mathcal{G}(A,B)$. For the sufficiency of the existence of non-trivial zero-divisors and units in torsion-free group rings, there are additional combinatorial properties and structures associated with the bipartite graph and the pairing. In this section, we will describe the additional combinatorial constructs and the conditions that are sufficient for the existence of counterexamples to the zero-divisor conjecture and the unit conjectures.

    Throughout this section $A :=\{a_1, \ldots, a_m\}$ and $B:=\{b_1, \ldots, b_n\}$, where $m$ and $n$ denote the cardinality of $A$ and $B$, respectively.
    
    \subsection{Product structures}
	Let $\mathcal{G}(A,B)$ denote the complete bipartite graph on two non-empty finite sets $A$ and $B$. 
    \subsubsection{Vertical edges}
    We will identify the edges in $\mathcal{G}(A,B)$, which we call the vertical edges, with the elements of the cartesian product $(a,b) \in A \times B$.  We denote the set of vertical edges of $\mathcal{G}(A,B)$ by $E_{\mathcal{G}}$ and the set of vertices by $V_{\mathcal{G}}$ which is $A \sqcup B$. 
    \subsubsection{Disjoint vertex condition}\label{subsect:disjoint_vertex} 
   
    Given two finite sets $A$ and $B$, a \textit{partition} \( P \) of the edge set \( A \times B \) is a collection of subsets \( P = \{C_i\}_{i \in I} \), where \( I \) is an index set, such that \( A \times B = \sqcup_{i \in I} C_i \)  where the subsets \( C_i \) are pairwise disjoint and their union equals \( A \times B \). Each subset \( C_i \subseteq A \times B \) (called a \textit{cell}) satisfies the disjoint vertex condition if for any two distinct edges \( (a, b), (a', b') \in C_i \), we have \( a \neq a' \) and \( b \neq b' \), i.e., no two edges in \( C_i \) share a common vertex. If \( |C_i| = k \), the cell \( C_i \) is called a \textit{\( k \)-cell}. A partition \( P = \{C_i\}_{i \in I} \), where \( I \) is an index set, satisfies the \textit{disjoint vertex condition} if each cell $C_i$ satisfies the disjoint vertex condition.
    
    \subsubsection{Product structures and substructures}
    Finally, a \textit{product structure} is a triple \( \Pi = (A, B, P) \), where:
    \begin{itemize}
        \item \( A \) and \( B \) are finite sets,
        \item \( P \) is a partition of \( A \times B \) satisfying the disjoint vertex condition.
    \end{itemize}
    
	By a \textit{subpartition} $S$ of the edge set $A\times B$, we mean $S =  \{C_i\}_{i \in I}$ where $I$ is an index set such that $C_i \subseteq A\times B$ and $\sqcup_{i \in I} C_i \subseteq A\times B$. Each $C_i$ is called a cell. A subpartition $S=\{C_i\}_{i \in I}$ where $I$ is an index set is said to satisfy the disjoint vertex condition if each cell $C_i$ satisfies the disjoint vertex condition. Given two finite sets $A, B$ and a subpartition $S$ satisfying the disjoint vertex condition, a \textit{product substructure} is a triple $(A,B,S)$.  
	
	Let $P$ be a partition or a subpartition. $P$ is called {\textit{even}} if each cell in $P$ is a $2$-cell. $P$ is called { \textit{odd}} if exactly one cell is $1$-cell and rest of the cells in $P$ are $2-$cells.
    
    
    Given a product structure or substructure $\Pi = (A, B, P)$, we will associate with $\Pi$ oriented horizontal edges, an oriented graph on $A$, denoted $\mathsf{L}_A$, and an oriented graph on $B$, denoted $\mathsf{L}_B$.

	\subsection{Horizontal edges} 
	
	Given $\Pi = (A,B,P)$, define the sets $\Bar{E}_A$ and $\Bar{E}_B$ as below:
	\begin{eqnarray*}
		\Bar{E}_A &:=& \{\{a, a'\}  ~|~  a, a'\in A \text{ and }\exists~ b, b' \in B \text{ such that }\{(a, b), (a', b')\} \in P\}\\
		\Bar{E}_B &:=& \{\{b, b'\}  ~|~  b, b' \in B \text{ and }\exists~ a, a' \in A \text{ such that }\{(a, b), (a', b')\} \in P\}\\
		\Bar{E}_{AB} &:=& \Bar{E}_A \sqcup \Bar{E}_B
	\end{eqnarray*} 
	
	The set $\Bar{E}_A$ is called the set of \textit{horizontal edges on $A$} with respect to $\Pi$ and the set $\Bar{E}_B$ is called  the set of \textit{horizontal edges on $B$} with respect to $\Pi$.
	\begin{example}\label{ex:horizontal_edges}
	    Let $A = \{a_1,a_2,a_3\}$ and $B=\{b_1,b_2,b_3,b_4\}$. Let $P = \{\{(a_1,b_1),(a_2,b_2)\}, \{(a_1,b_2), (a_2,b_3)\}\}$ be a subpartition of $A\times B $, then $\Bar{E}_A = \{\{a_1,a_2\}\}$ and $\Bar{E}_B := \{\{b_1,b_2\}, \{b_2,b_3\}\}$.
	\end{example}

	\subsection{Orientation of horizontal edges}
 
	Define $E_A$ and $E_B$ as below:
    \begin{align}
    E_A & :=  \{(a, a')  \mid a,a'\in A \text{ and }\exists~ b, b' \in B \text{ such that } \{(a, b), (a', b')\} \in P\} \label{def:Ha} \\
    E_B & :=  \{(b, b')  \mid b,b'\in B \text{ and }\exists~ a, a' \in A \text{ such that } \{(a, b), (a', b')\} \in P\} \label{def:Hb} \\
    E_{AB} & :=  E_A \sqcup E_B \label{def:Hab}
\end{align}
For instance, in Example \ref{ex:horizontal_edges}, $E_A = \{(a_1, a_2), (a_2, a_1)\}$ and $E_B = \{(b_1, b_2),(b_2,b_3), (b_2, b_1), (b_3, b_2)\}$.
    \begin{definition}[Orientation]\label{def:orientation}
        \textit{An orientation} on $\Pi$ is a function $O: \Bar{E}_{AB} \to E_{AB}$ such that
	\begin{itemize}
	    \item for each $\{a,a'\}\in \Bar{E}_A, O(\{a,a'\}) = (a,a')$ or $O(\{a,a'\}) = (a',a)$,
	    \item for each $\{b,b'\}\in \Bar{E}_B, O(\{b,b'\}) = (b,b')$ or $O(\{b,b'\}) = (b',b)$, and 
     \item for each $2-$cell $\{(a,b),(a',b')\} \in P$,
     \begin{align*}
         ( O(\{a,a'\}) &= (a,a') \text{ and } O(\{b,b'\}) = (b,b')), \text{ or, }  \\
         (O(\{a,a'\}) &= (a',a) \text{ and } O(\{b,b'\}) = (b',b))
     \end{align*}
	\end{itemize}
    \end{definition}

    \begin{definition}[Orientable product substructure]\label{def:orientablepi}
        A product structure or substructure $\Pi = (A,B,P)$ is called \text{orientable} if there exists an orientation on $\Pi$. 
    \end{definition}
    For instance, in Example \ref{ex:horizontal_edges}, $O(\{a_1,a_2\}) = (a_1,a_2)$ and $O(\{b_1,b_2\}) = (b_1,b_2)$. This implies $O(\{b_2,b_3\}) = (b_2,b_3)$.
    Note that if a product substructure $\Pi = (A,B,S)$ is not orientable then for any product substructure $\Pi' = (A,B,S')$ such that $S \subset S'$, $\Pi'$ is not orientable as well.

    \begin{example}[Failure of orientability]\label{ex:2x2_not_possible}
        Let $A$ and $B$ be two finite sets. Let $C_1 = \{(a_1,b_1),(a_2,b_2)\}$ and $C_2 = \{(a_1,b_2),(a_2,b_1)\}$ as illustrated in Figure \ref{fig:orientation_failure}. Consider $\Pi = (A,B, P)$ where $P = \{C_1, C_2\}$. For $C_1$, the horizontal edges are $\{a_1, a_2\}$ and $\{b_1, b_2\}$. Let $O(\{a_1, a_2\}) = (a_1, a_2)$. Then $O(\{b_1, b_2\}) = (b_1, b_2)$ by Definition \ref{def:orientation}. Since $C_2$ consists of horizontal edges $\{a_1, a_2\}$ and $\{b_1, b_2\}$ as well, $O(\{b_1, b_2\}) = (b_2, b_1)$ which contradicts orientability. By a similar argument, $O(\{a_1, a_2\}) = (a_2, a_1)$ yields failure of orientability. This also proves that there is no orientable even product structure for $A =\{a_1,a_2\}$ and $B = \{b_1,b_2\}$.
    \end{example}

    \subsection{Horizontal graphs}

     To an orientable product substructure $\Pi$, we associate two directed graphs, $\mathsf{L}_A$ and $\mathsf{L}_B$, on the sets of vertices $A$ and $B$, respectively. The edges of $\mathsf{L}_A$ are determined by $E_A$: if $(a, a') \in E_A$, there is a directed edge from $a$ to $a'$. Similarly, the edges of $\mathsf{L}_B$ are determined by $E_B$: if $(b, b') \in E_B$, there is a directed edge from $b$ to $b'$. We then define the combined graph $\mathsf{L}_{AB}$ as the disjoint union of $\mathsf{L}_A$ and $\mathsf{L}_B$:
    \begin{equation}\label{eqn:defLAB}
    \mathsf{L}_{AB} := \mathsf{L}_A \sqcup \mathsf{L}_B.    
    \end{equation}

The following example illustrates the failure of orientability for a product substructure $\Pi$. Let $A$ and $B$ be two finite sets. Consider the product structure $\Pi = (A, B, P)$, where $P := \{C_1, C_2\}$ and
    \[
    C_1 := \{(a_1, b_1), (a_2, b_2)\}, \quad C_2 := \{(a_1, b_2), (a_2, b_1)\}.
    \]
    For the cell $C_1$, the horizontal edges are $\{a_1, a_2\}$ and $\{b_1, b_2\}$. Assume the orientation function $O$ satisfies $O(\{a_1, a_2\}) := (a_1, a_2)$. Then, by Definition~\ref{def:orientation}, $O(\{b_1, b_2\}) = (b_1, b_2)$ to maintain consistent orientation for the horizontal edges of $C_1$. Next, consider $C_2$. The horizontal edges of $C_2$ are also $\{a_1, a_2\}$ and $\{b_1, b_2\}$. To satisfy the orientation condition for $C_2$, we would need $O(\{b_1, b_2\}) = (b_2, b_1)$ to match the edge configuration of $C_2$. This leads to a contradiction, as $O(\{b_1, b_2\})$ cannot simultaneously be $(b_1, b_2)$ for $C_1$ and $(b_2, b_1)$ for $C_2$. A similar contradiction arises if we assume $O(\{a_1, a_2\}) := (a_2, a_1)$ instead. Thus, the product structure $\Pi$ is not orientable.
   \begin{figure}
       \centering




\tikzstyle arrowstyle=[scale=1]
\tikzstyle arrowtipinmiddle=[postaction={decorate,decoration={markings,mark=at position .56 with {\arrow[arrowstyle]{stealth'}}}}]
\tikzstyle verticalarrowtipinmiddle=[postaction={decorate,decoration={markings,mark=at position .7 with {\arrow[arrowstyle]{stealth'}}}}]
\tikzstyle 2arrowtipinmiddle=[postaction={decorate,decoration={markings,mark=at position .53 with {\arrow[arrowstyle]{{stealth'}}},mark=at position .58 with {\arrow[arrowstyle]{{stealth'}}}}}]
\tikzstyle vertical2arrowtipinmiddle=[postaction={decorate,decoration={markings,mark=at position .59 with {\arrow[arrowstyle]{{stealth'}}},mark=at position .62 with {\arrow[arrowstyle]{{stealth'}}}}}]
\tikzstyle 3arrowtipinmiddle=[postaction={decorate,decoration={markings,mark=at position .47 with {\arrow[arrowstyle]{stealth'}},mark=at position .56 with {\arrow[arrowstyle]{stealth'}},mark=at position .65 with {\arrow[arrowstyle]{stealth'}}}}]
\tikzstyle vertical3arrowtipinmiddle=[postaction={decorate,decoration={markings,mark=at position .33 with {\arrow[arrowstyle]{stealth'}},mark=at position .36 with {\arrow[arrowstyle]{stealth'}},mark=at position .39 with {\arrow[arrowstyle]{stealth'}}}}]
\tikzstyle trianglearrowtipinmiddle=[postaction={decorate,decoration={markings,mark=at position .56 with {\arrow[arrowstyle]{Triangle[open]}}}}]
\tikzstyle verticatrianglelarrowtipinmiddle=[postaction={decorate,decoration={markings,mark=at position .7 with {\arrow[arrowstyle]{Triangle[open]}}}}]
 \tikzstyle triangle2arrowtipinmiddle=[postaction={decorate,decoration={markings,mark=at position .53 with {\arrow[arrowstyle]{{Triangle[open]}}},mark=at position .58 with {\arrow[arrowstyle]{{Triangle[open]}}}}}]
\tikzstyle verticaltriangle2arrowtipinmiddle=[postaction={decorate,decoration={markings,mark=at position .59 with {\arrow[arrowstyle]{{Triangle[open]}}},mark=at position .62 with {\arrow[arrowstyle]{{Triangle[open]}}}}}]
\tikzstyle triangle3arrowtipinmiddle=[postaction={decorate,decoration={markings,mark=at position .47 with {\arrow[arrowstyle]{Triangle[open]}},mark=at position .56 with {\arrow[arrowstyle]{Triangle[open]}},mark=at position .65 with {\arrow[arrowstyle]{Triangle[open]}}}}]
\tikzstyle verticaltriangle3arrowtipinmiddle=[postaction={decorate,decoration={markings,mark=at position .33 with {\arrow[arrowstyle]{Triangle[open]}},mark=at position .36 with {\arrow[arrowstyle]{Triangle[open]}},mark=at position .39 with {\arrow[arrowstyle]{Triangle[open]}}}}]

\begin{tikzpicture}

    \node[fill=black, circle, inner sep=1.5pt, label=above:$b_1$] (b1) at (0,3) {};
    \node[fill=black, circle, inner sep=1.5pt, label=above:\small{$b_2$}] (b2) at (1.2,3) {};
    \node[fill=none, text centered] (n) at (1.7,3) {$\ldots$};
    \node[fill=black, circle, inner sep=1.5pt, label=above:$b_i$] (bn) at (2.2,3) {};
    \node[fill=black, circle, inner sep=1.5pt, label=below:$a_1$] (a1) at (0,1) {};
    \node[fill=black, circle, inner sep=1.5pt, label=below:$a_2$] (a2) at (1.2,1) {};
    \node[fill=none, text centered] (n) at (1.9,1) {$\ldots$};
    \node[fill=black, circle, inner sep=1.5pt, label=below:$a_m$] (am) at (2.5,1) {};
    \draw[line width=0.0500cm, color=royalblue] (a1) -- (b1);
    \draw[line width=0.0500cm, color=royalblue] (a2) -- (b2);
    \draw[arrowtipinmiddle, line width=0.0500cm, color=royalblue, bend right] (a1) to (a2);
    \draw[arrowtipinmiddle, line width=0.0500cm, color=royalblue, bend left] (b1) to (b2);
    
    \node[fill=black, circle, inner sep=1.5pt, label=above:$b_1$] (b1) at (5.5,3) {};
    \node[fill=black, circle, inner sep=1.5pt, label=above:\small{$b_2$}] (b2) at (7.5,3) {};
    \node[fill=none, text centered] (n) at (8,3) {$\ldots$};
    \node[fill=black, circle, inner sep=1.5pt, label=above:$b_n$] (bn) at (8.75,3) {};
    \node[fill=black, circle, inner sep=1.5pt, label=below:$a_1$] (a1) at (5.5,1) {};
    \node[fill=black, circle, inner sep=1.5pt, label=below:$a_2$] (a2) at (7,1) {};
    \node[fill=none, text centered] (n) at (7.8,1) {$\ldots$};
    \node[fill=black, circle, inner sep=1.5pt, label=below:$a_m$] (am) at (8.5,1) {};
    \draw[line width=0.0500cm, color=royalblue] (a1) -- (b1);
    \draw[line width=0.0500cm, color=royalblue] (a2) -- (b2);
    \draw[line width=0.0500cm, color=royalblue] (a1) -- (b2);
    \draw[line width=0.0500cm, color=royalblue] (a2) -- (b1);
   
    \draw[arrowtipinmiddle, line width=0.0500cm, color=royalblue, bend right] (a1) to (a2);
    \draw[arrowtipinmiddle, line width=0.0500cm, color=royalblue, bend left=20] (b1) to (b2);
    \draw[arrowtipinmiddle, line width=0.0500cm, color=royalblue, bend right = 60] (b2) to (b1);
   
    \node[fill=none, text centered] (n) at (6.8,0) {\footnotesize{$P = \{C_1, C_2\}$}};    
     \node[fill=none, text centered] (n) at (1,0) {\footnotesize{$C_1$}};
    \end{tikzpicture}
       \caption{Left: For $C_1$, $O(\{a_1, a_2\}) = (a_1, a_2)$ and $O(\{b_1, b_2\}) = (b_1, b_2)$. Right: For $C_2$, $O(\{a_1, a_2\}) = (a_1, a_2)$ and $O(\{b_1, b_2\}) = (b_2, b_1)$, showing that $\Pi = (A,B,P)$ is not orientable.}

       \label{fig:orientation_failure}
   \end{figure}
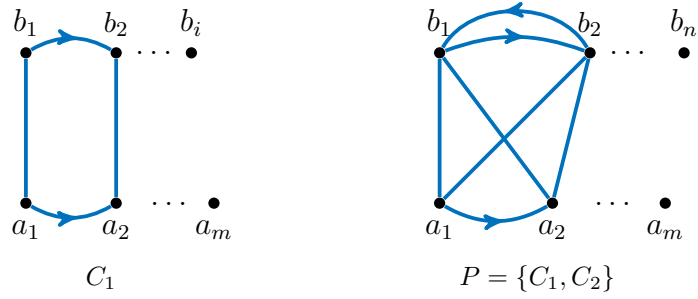

    For an orientable product structure $\Pi = (A, B, P)$, the associated graph $\mathsf{L}_{AB}$ cannot contain multiple edges between any pair of vertices.

    In this paper, we will be interested in product substructures (or structures) $\Pi=(A,B,P)$ where $P$ is either even or odd.  Throughout the paper by an edge $(x,y)$ we mean a directed edge from vertex $x$ to vertex $y$ and by an edge $\{x,y\}$ we mean an undirected edge between the vertices $x$ and $y$.

    \subsection{Taiko, the product graph}
    To a product substructure, we will associate two graphs: taiko $\mathcal{T}(\Pi)$ and middle-link graph $L_1(\Pi)$.
    
    \begin{definition}[Taiko]\label{def:taiko}
    Let \( \Pi = (A, B, P) \) be a product substructure, where \( P = \{C_i\}_{i\in I} \) is either even or odd subpartition of \( A \times B \) indexed by \( I \). Let \( E(P) = \{(a, b) \mid \exists C_i \in P \text{ so that }(a, b) \in C_i\} \subset A \times B \), representing a subset of vertical edges of the bipartite graph \( \mathcal{G}(A, B) \). 

    \textit{A} \textit{taiko}, also called \textit{a product graph}, denoted as ${\mathcal{T}}{(\Pi)}$, is a directed graph constructed from the product structure \( \Pi \), with the following components:
    \begin{itemize}
        \item The set of vertices is the disjoint union of \( A \) and \( B \), i.e., \( V(\mathcal{T}) = A \sqcup B \),
        \item The edge set is \( E_{\mathcal{T}} = E_{AB} \sqcup E \), where:
        \begin{itemize}
            \item \( E_{AB} \) is the set of the horizontal edges of $\mathsf{L}_A \sqcup \mathsf{L}_B$ as given in \eqref{def:Hab},
            \item \( E(P) \) is the set of vertical edge set as defined earlier.
        \end{itemize}
    \end{itemize}
\end{definition}

 \subsection{Equivalence classes}
    We define a relation $\sim$ on the set $\Bar{E}_{AB}$ as $\{a,a'\} \sim \{b,b'\}$ if there exists a cell $C$ of the form $\{(a,b), (a',b')\}$ or $\{(a,b'), (a',b)\}$. This is an equivalence relation and we color the edges of each equivalence class with one color. To the cell $C$ we assign the same color as that of horizontal edge $\{a,a'\}$ or $\{b,b'\}$. Similarly, for oriented product substructures, $\Pi$, an edge $(a,a')$ (or $(b,b')$) inherits the same color as the edge $\{a,a'\}$ (or $\{b,b'\}$). Thus this relation is an equivalence relation on $E_{AB}$ as well.

    \begin{convention}[colors and directions in a taiko]\label{convention:taiko_verticaledges_etc}
        A taiko \( \mathcal{T}(\Pi) \) is illustrated by placing \( A \) at the bottom and \( B \) at the top. A vertical edge \( (a, b) \in A \times B \) is drawn whenever there exists a 2-cell in \( P \) containing \( (a, b) \). All horizontal edges from the set \( E_{AB} \) are also included. The orientation of vertical edges is typically omitted in the illustration, as it does not affect the theoretical structure of the taiko. Moreover, with the definition of the taiko, the 2-cells in $P$ naturally inherit the same colors as their horizontal edges have, that is, there is a consistent coloring of 2-cells and horizontal edges.
    \end{convention}
    
    \begin{example}\label{ex:taiko}
        
    Let $A=\{a_1,a_2,a_3,a_4\}$, $B=\{b_1,b_2,b_3,b_4\}$ and let $P = \bigsqcup_{i=1}^{8} \{C_i\}$ be an even partition such that $C_1={\color{orange}{\mathbf{\{(a_1,b_1),(a_2,b_2)\}}}}$ , $C_2={\color{springgreen}{\mathbf{\{(a_1,b_2),(a_3,b_3)\}}}}$ , $C_3={\color{orange}{\mathbf{\{(a_2,b_1),(a_3,b_2)\}}}}$ , $C_4={\color{royalblue}{\mathbf{\{(a_1,b_3),(a_4,b_4)\}}}}$ , $C_5={\color{lavender}{\mathbf{\{(a_2,b_3),(a_4,b_1)\}}}}$ , $C_6={\color{springgreen}{\mathbf{\{(a_1,b_4),(a_3,b_1)\}}}}$ , $C_7={\color{lavender}{\mathbf{\{(a_2,b_4),(a_4,b_2)\}}}}$ , $C_8={\color{royalblue}{\mathbf{\{(a_4,b_3),(a_3,b_4)\}}}}$. For each $2$-cell $C_i$, no two edges share a common vertex, that is, if $C_i = \{(a,b),(a',b')\}$ then $a\neq a'$ and $b\neq b'$. Each $C_i$ satisfies the disjoint vertex condition in Section \ref{subsect:disjoint_vertex} and hence, $P$ satisfies the disjoint vertex condition. This gives us a product structure $\Pi = (A,B, P)$. We can visualize $\Pi$ as the bipartite graph $\mathcal{G}(A,B)$ such that each cell corresponds to a pair of edges of this bipartite graph $\mathcal{G}(A,B)$ as shown in Figure \ref{fig:taiko_example_4_by_4}.

\begin{figure}[ht]
\tikzstyle arrowstyle=[scale=1]
\tikzstyle arrowtipinmiddle=[postaction={decorate,decoration={markings,mark=at position .56 with {\arrow[arrowstyle]{stealth'}}}}]
\tikzstyle verticalarrowtipinmiddle=[postaction={decorate,decoration={markings,mark=at position .7 with {\arrow[arrowstyle]{stealth'}}}}]
\tikzstyle 2arrowtipinmiddle=[postaction={decorate,decoration={markings,mark=at position .53 with {\arrow[arrowstyle]{{stealth'}}},mark=at position .58 with {\arrow[arrowstyle]{{stealth'}}}}}]
\tikzstyle vertical2arrowtipinmiddle=[postaction={decorate,decoration={markings,mark=at position .59 with {\arrow[arrowstyle]{{stealth'}}},mark=at position .62 with {\arrow[arrowstyle]{{stealth'}}}}}]
\tikzstyle 3arrowtipinmiddle=[postaction={decorate,decoration={markings,mark=at position .47 with {\arrow[arrowstyle]{stealth'}},mark=at position .56 with {\arrow[arrowstyle]{stealth'}},mark=at position .65 with {\arrow[arrowstyle]{stealth'}}}}]
\tikzstyle vertical3arrowtipinmiddle=[postaction={decorate,decoration={markings,mark=at position .33 with {\arrow[arrowstyle]{stealth'}},mark=at position .36 with {\arrow[arrowstyle]{stealth'}},mark=at position .39 with {\arrow[arrowstyle]{stealth'}}}}]
\tikzstyle trianglearrowtipinmiddle=[postaction={decorate,decoration={markings,mark=at position .56 with {\arrow[arrowstyle]{Triangle[open]}}}}]
\tikzstyle verticatrianglelarrowtipinmiddle=[postaction={decorate,decoration={markings,mark=at position .7 with {\arrow[arrowstyle]{Triangle[open]}}}}]
\tikzstyle triangle2arrowtipinmiddle=[postaction={decorate,decoration={markings,mark=at position .53 with {\arrow[arrowstyle]{{Triangle[open]}}},mark=at position .58 with {\arrow[arrowstyle]{{Triangle[open]}}}}}]
\tikzstyle verticaltriangle2arrowtipinmiddle=[postaction={decorate,decoration={markings,mark=at position .59 with {\arrow[arrowstyle]{{Triangle[open]}}},mark=at position .62 with {\arrow[arrowstyle]{{Triangle[open]}}}}}]
\tikzstyle triangle3arrowtipinmiddle=[postaction={decorate,decoration={markings,mark=at position .47 with {\arrow[arrowstyle]{Triangle[open]}},mark=at position .56 with {\arrow[arrowstyle]{Triangle[open]}},mark=at position .65 with {\arrow[arrowstyle]{Triangle[open]}}}}]
\tikzstyle verticaltriangle3arrowtipinmiddle=[postaction={decorate,decoration={markings,mark=at position .33 with {\arrow[arrowstyle]{Triangle[open]}},mark=at position .36 with {\arrow[arrowstyle]{Triangle[open]}},mark=at position .39 with {\arrow[arrowstyle]{Triangle[open]}}}}]
\begin{tikzpicture}[scale = 0.5]
\draw[fill=white] (0.0000cm,0cm) circle(0.1400cm) node[] {};
\draw (0.0000cm,-0.8cm) node[] {${\scriptstyle a_1}$};
\draw[fill=white] (7.3333cm,0cm) circle(0.1400cm) node[] {};
\draw (7.3333cm,-0.8cm) node[] {${\scriptstyle a_2}$};
\draw[fill=white] (14.6667cm,0cm) circle(0.1400cm) node[] {};
\draw (14.6667cm,-0.8cm) node[] {${\scriptstyle a_3}$};
\draw[fill=white] (22.0000cm,0cm) circle(0.1400cm) node[] {};
\draw (22.0000cm,-0.8cm) node[] {${\scriptstyle a_4}$};
\filldraw[black] (0.0000cm,8.0000cm) circle(0.1300cm) node[] {};
\draw (0.0000cm,8.8000cm) node[] {${\scriptstyle b_1}$};
\filldraw[black] (7.3333cm,8.0000cm) circle(0.1300cm) node[] {};
\draw (7.3333cm,8.8000cm) node[] {${\scriptstyle b_2}$};
\filldraw[black] (14.6667cm,8.0000cm) circle(0.1300cm) node[] {};
\draw (14.6667cm,8.8000cm) node[] {${\scriptstyle b_3}$};
\filldraw[black] (22.0000cm,8.0000cm) circle(0.1300cm) node[] {};
\draw (22.0000cm,8.8000cm) node[] {${\scriptstyle b_4}$};
\draw[verticalarrowtipinmiddle,line width=0.0500cm,color=orange] (0.0000cm,0cm) to (0.0000cm,8.0000cm);
\draw[verticalarrowtipinmiddle,line width=0.0500cm,color=orange] (7.3333cm,0cm) to (7.3333cm,8.0000cm);
\draw[verticalarrowtipinmiddle,line width=0.0500cm,color=springgreen] (0.0000cm,0cm) to (7.3333cm,8.0000cm);
\draw[verticalarrowtipinmiddle,line width=0.0500cm,color=springgreen] (14.6667cm,0cm) to (14.6667cm,8.0000cm);
\draw[verticalarrowtipinmiddle,line width=0.0500cm,color=orange] (7.3333cm,0cm) to (0.0000cm,8.0000cm);
\draw[verticalarrowtipinmiddle,line width=0.0500cm,color=orange] (14.6667cm,0cm) to (7.3333cm,8.0000cm);
\draw[verticalarrowtipinmiddle,line width=0.0500cm,color=royalblue] (0.0000cm,0cm) to (14.6667cm,8.0000cm);
\draw[verticalarrowtipinmiddle,line width=0.0500cm,color=royalblue] (22.0000cm,0cm) to (22.0000cm,8.0000cm);
\draw[verticalarrowtipinmiddle,line width=0.0500cm,color=lavender] (7.3333cm,0cm) to (14.6667cm,8.0000cm);
\draw[verticalarrowtipinmiddle,line width=0.0500cm,color=lavender] (22.0000cm,0cm) to (0.0000cm,8.0000cm);
\draw[verticalarrowtipinmiddle,line width=0.0500cm,color=springgreen] (0.0000cm,0cm) to (22.0000cm,8.0000cm);
\draw[verticalarrowtipinmiddle,line width=0.0500cm,color=springgreen] (14.6667cm,0cm) to (0.0000cm,8.0000cm);
\draw[verticalarrowtipinmiddle,line width=0.0500cm,color=lavender] (7.3333cm,0cm) to (22.0000cm,8.0000cm);
\draw[verticalarrowtipinmiddle,line width=0.0500cm,color=lavender] (22.0000cm,0cm) to (7.3333cm,8.0000cm);
\draw[verticalarrowtipinmiddle,line width=0.0500cm,color=royalblue] (22.0000cm,0cm) to (14.6667cm,8.0000cm);
\draw[verticalarrowtipinmiddle,line width=0.0500cm,color=royalblue] (14.6667cm,0cm) to (22.0000cm,8.0000cm);
\draw[arrowtipinmiddle, line width=0.0500cm, color=orange, bend right=20] (0.0000cm,0cm) to (7.3333cm,0cm);
\draw[arrowtipinmiddle, line width=0.0500cm, color=orange, bend left=20] (0.0000cm,8.0000 cm) to (7.3333cm,8.0000cm);
\draw[arrowtipinmiddle, line width=0.0500cm, color=springgreen, bend right] (0.0000cm,0cm) to (14.6667cm,0cm);
\draw[arrowtipinmiddle, line width=0.0500cm, color=springgreen, bend left=20] (7.3333cm,8.0000 cm) to (14.6667cm,8.0000cm);
\draw[arrowtipinmiddle, line width=0.0500cm, color=orange, bend right=20] (7.3333cm,0cm) to (14.6667cm,0cm);
\draw[arrowtipinmiddle, line width=0.0500cm, color=royalblue, bend right] (0.0000cm,0cm) to (22.0000cm,0cm);
\draw[arrowtipinmiddle, line width=0.0500cm, color=lavender, bend right] (7.3333cm,0cm) to (22.0000cm,0cm);
\draw[arrowtipinmiddle, line width=0.0500cm, color=lavender, bend right] (14.6667cm,8.0000 cm) to (0.0000cm,8.0000cm);
\draw[arrowtipinmiddle, line width=0.0500cm, color=springgreen, bend right] (0.0000cm,0cm) to (14.6667cm,0cm);
\draw[arrowtipinmiddle, line width=0.0500cm, color=springgreen, bend right] (22.0000cm,8.0000 cm) to (0.0000cm,8.0000cm);
\draw[arrowtipinmiddle, line width=0.0500cm, color=lavender, bend right] (7.3333cm,0cm) to (22.0000cm,0cm);
\draw[arrowtipinmiddle, line width=0.0500cm, color=lavender, bend right] (22.0000cm,8.0000 cm) to (7.3333cm,8.0000cm);
\draw[arrowtipinmiddle, line width=0.0500cm, color=royalblue, bend left=20] (22.0000cm,0cm) to (14.6667cm,0cm);
\draw[arrowtipinmiddle, line width=0.0500cm, color=royalblue, bend left=20] (14.6667cm,8.0000 cm) to (22.0000cm,8.0000cm);
\draw[fill=white] (0.0000cm,0cm) circle(0.1400cm) node[below] {};
\draw[fill=white] (7.3333cm,0cm) circle(0.1400cm) node[below] {};
\draw[fill=white] (14.6667cm,0cm) circle(0.1400cm) node[below] {};
\draw[fill=white] (22.0000cm,0cm) circle(0.1400cm) node[below] {};
\filldraw[black] (0.0000cm,8.0000cm) circle(0.1300cm) node[above] {};
\filldraw[black] (7.3333cm,8.0000cm) circle(0.1300cm) node[above] {};
\filldraw[black] (14.6667cm,8.0000cm) circle(0.1300cm) node[above] {};
\filldraw[black] (22.0000cm,8.0000cm) circle(0.1300cm) node[above] {};
\end{tikzpicture}
\caption{$(m,n) = (4,4)$. Taiko (the product graph). The 8 \ 2-cells of the partition are split into 4 colors: orange, green, blue and purple}
\label{fig:taiko_example_4_by_4}
\end{figure}
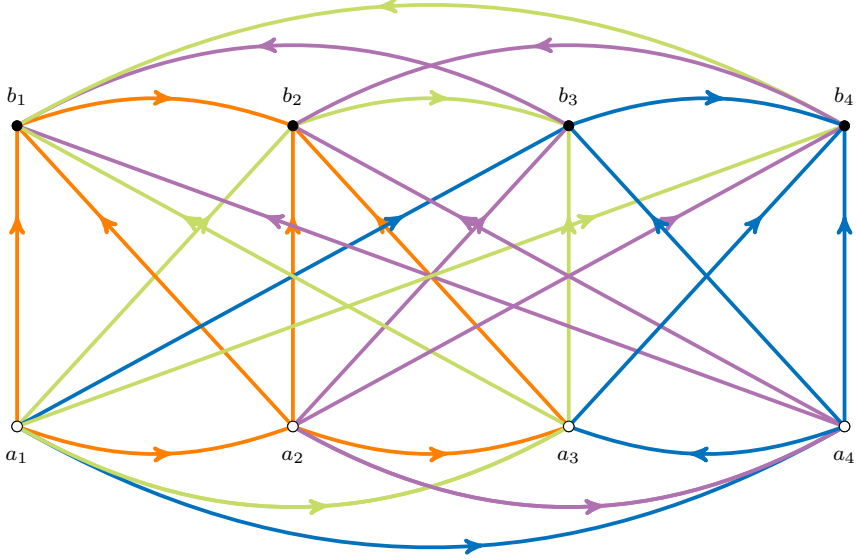

We now check orientability of the product structure $\Pi = (A,B,P)$. Let $O(\{a_1,a_2\}):=(a_1,a_2)$ then $O(\{b_1,b_2\}) = (b_1, b_2)$. From cell $C_3$, we have $O(\{a_2,a_3\}) = (a_2,a_3)$. Similarly, let $O(\{a_1,a_3\}) := (a_1,a_3)$ then $O(\{b_2,b_3\}) = (b_2,b_3)$. Proceeding in the similar fashion, we see that $\Pi$ is orientable. The edge set of $\mathsf{L}_A$ is $\{(a_1,a_2),(a_2,a_3),(a_1,a_3),(a_1,a_4),(a_2,a_4),(a_3,a_4)\}$ and the edge set of $\mathsf{L}_B$ is $\{(b_1,b_2),(b_1,b_3),(b_1,b_4),(b_2,b_3),(b_2,b_4),(b_3,b_4)\}$. Both $\mathsf{L}_A$ and $\mathsf{L}_B$ are complete graphs on their respective set of vertices. A product graph, or a taiko, corresponding to the orientable product structure $\Pi$ is shown in Figure \ref{fig:taiko_example_4_by_4}.

The equivalence classes of \( E_{AB} \) are:
    \begin{align*}
         \boldsymbol{\color{orange}{\{(a_1, a_2), (a_2, a_3),(b_1, b_2)\}}}, \, \boldsymbol{\color{springgreen}{\{(a_1, a_3),(b_4, b_1), (b_2, b_3)\}}}, \, \\\boldsymbol{\color{lavender}{\{(a_2, a_4),(b_3, b_1), (b_4, b_2)\}}}, \, \boldsymbol{\color{royalblue}{\{(a_1, a_4), (a_3, a_4),(b_3, b_4)\}}}
    \end{align*}  
    illustrated as orange, green, purple, and blue in Figure \ref{fig:taiko_example_4_by_4}. 
    
\end{example}

\subsection{The Middle Link}
    For a product substructure \(\Pi=(A,B,P)\) with the disjoint vertex condition, Mineyev associates a \(2\)-complex \(Y_\Pi\) whose \(0\)-skeleton consists of three vertices $x_A$, $x_1$ and $x_B$. For definitions and further details, see \cite{Mineyev2024}. We will only use the links of these three vertices; in particular,
    \[
    \operatorname{Lk}_{Y_\Pi}(x_A)\cong \mathsf L_A,\qquad
    \operatorname{Lk}_{Y_\Pi}(x_B)\cong \mathsf L_B .
    \]  
    \begin{definition}[Middle link]\label{def:middle-link}
    The \textit{middle link} of \(\Pi\) is
    \(
    \mathsf L_1(\Pi)\ :=\ \operatorname{Lk}_{Y_\Pi}(x_1).
    \) When \(\Pi\) is clear, we simply write \(\mathsf L_1\).
    \end{definition}
    
    
    \subsubsection*{Combinatorial construction of $\mathsf{L}_1$}
    Let \(\kappa:E_{AB}\to\{1,\dots,k\}\) be the color map on the directed horizontal edges of the taiko \(\mathcal T(\Pi)\) (equivalently, the equivalence-class map on \(E_{AB}\)); we refer to the \(i\)-th color and the \(i\)-th class interchangeably.

    \begin{itemize}
      \item \textit{Middle vertices.} For each color \(i\in\{1,\dots,k\}\), introduce two vertices
      \((i,\mathrm{in})\) and \((i,\mathrm{out})\). Let
      \[
        M:= \{1,\dots,k\} \times \{\text{in}, \text{out}\}.
      \]
    
      \item \textit{Vertex set.} \(V(\mathsf L_1)=A\ \sqcup\ B\ \sqcup\ M\).
    
      \item \textit{Edge set.} For each directed horizontal edge \(e=(x,y)\in E_{AB}\) of
      \(\mathcal T(\Pi)\) with color \(\kappa(e)=i\), add the following (undirected) edges to \(E(\mathsf L_1)\) 
      \[
      \{x,(i,\mathrm{out})\}\quad\text{and}\quad \{y,(i,\mathrm{in})\}.
      \]
    \end{itemize}
    
    Thus \(\mathsf L_1\) is an undirected graph on \(A\sqcup B\sqcup M\). For figures, we place \(A\) at the bottom, \(B\) at the top, and the middle vertices \(M\) between them.
    \magenta{(Any fixed orientation of the horizontal edges \(E_{AB}\) yields an isomorphic \(\mathsf L_1\); we use the
    natural forward orientations on \(\mathsf L_A\) and \(\mathsf L_B\).)}

    \begin{example*}
        For the product structure and taiko in Example \ref{ex:taiko}, we have $4$ colors and hence for $L_1$, the vertex set is $A \ \sqcup \ B \ \sqcup \ M$ where 
    \begin{align*}
        M = \{(\text{orange},{\text{in}}), (\text{orange},{\text{out}}), (\text{green},{\text{in}}), (\text{green},{\text{out}}), \\
        (\text{purple},{\text{in}}),
        (\text{purple},{\text{out}}), (\text{blue},{\text{in}}), (\text{blue},{\text{out}})\}
    \end{align*}

    Figure \ref{fig:taiko_example_middle_link} illustrates the middle link $\mathsf{L}_1$. 
    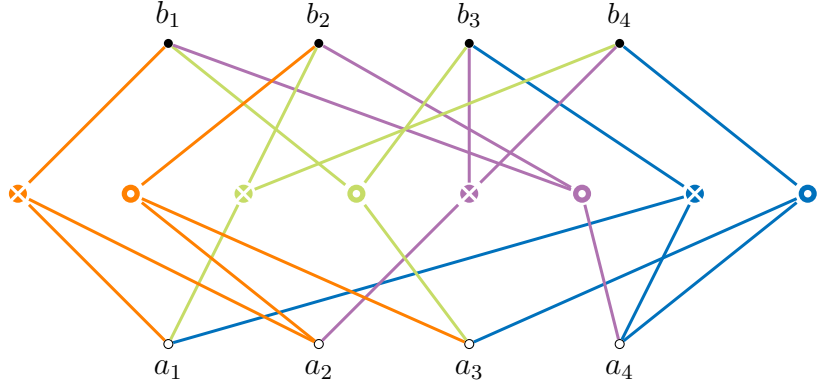
\begin{figure}[]    
        \centering
        \tikzstyle arrowstyle=[scale=1]
    \tikzstyle arrowtipinmiddle=[postaction={decorate,decoration={markings,mark=at position .56 with {\arrow[arrowstyle]{stealth'}}}}]
    \tikzstyle verticalarrowtipinmiddle=[postaction={decorate,decoration={markings,mark=at position .7 with {\arrow[arrowstyle]{stealth'}}}}]
    \tikzstyle 2arrowtipinmiddle=[postaction={decorate,decoration={markings,mark=at position .53 with {\arrow[arrowstyle]{{stealth'}}},mark=at position .58 with {\arrow[arrowstyle]{{stealth'}}}}}]
    \tikzstyle vertical2arrowtipinmiddle=[postaction={decorate,decoration={markings,mark=at position .59 with {\arrow[arrowstyle]{{stealth'}}},mark=at position .62 with {\arrow[arrowstyle]{{stealth'}}}}}]
    \tikzstyle 3arrowtipinmiddle=[postaction={decorate,decoration={markings,mark=at position .47 with {\arrow[arrowstyle]{stealth'}},mark=at position .56 with {\arrow[arrowstyle]{stealth'}},mark=at position .65 with {\arrow[arrowstyle]{stealth'}}}}]
    \tikzstyle vertical3arrowtipinmiddle=[postaction={decorate,decoration={markings,mark=at position .33 with {\arrow[arrowstyle]{stealth'}},mark=at position .36 with {\arrow[arrowstyle]{stealth'}},mark=at position .39 with {\arrow[arrowstyle]{stealth'}}}}]
    \tikzstyle trianglearrowtipinmiddle=[postaction={decorate,decoration={markings,mark=at position .56 with {\arrow[arrowstyle]{Triangle[open]}}}}]
    \tikzstyle verticatrianglelarrowtipinmiddle=[postaction={decorate,decoration={markings,mark=at position .7 with {\arrow[arrowstyle]{Triangle[open]}}}}]
     \tikzstyle triangle2arrowtipinmiddle=[postaction={decorate,decoration={markings,mark=at position .53 with {\arrow[arrowstyle]{{Triangle[open]}}},mark=at position .58 with {\arrow[arrowstyle]{{Triangle[open]}}}}}]
    \tikzstyle verticaltriangle2arrowtipinmiddle=[postaction={decorate,decoration={markings,mark=at position .59 with {\arrow[arrowstyle]{{Triangle[open]}}},mark=at position .62 with {\arrow[arrowstyle]{{Triangle[open]}}}}}]
    \tikzstyle triangle3arrowtipinmiddle=[postaction={decorate,decoration={markings,mark=at position .47 with {\arrow[arrowstyle]{Triangle[open]}},mark=at position .56 with {\arrow[arrowstyle]{Triangle[open]}},mark=at position .65 with {\arrow[arrowstyle]{Triangle[open]}}}}]
    \tikzstyle verticaltriangle3arrowtipinmiddle=[postaction={decorate,decoration={markings,mark=at position .33 with {\arrow[arrowstyle]{Triangle[open]}},mark=at position .36 with {\arrow[arrowstyle]{Triangle[open]}},mark=at position .39 with {\arrow[arrowstyle]{Triangle[open]}}}}]
    \begin{tikzpicture}

    \node[fill=black, circle, inner sep=1.2pt, label=above:$b_1$] (b1) at (2,4) {};
    \node[fill=black, circle, inner sep=1.2pt, label=above:\small{$b_2$}] (b2) at (4,4) {};
    \node[fill=black, circle, inner sep=1.2pt, label=above:\small{$b_3$}] (b3) at (6,4) {};
    \node[fill=black, circle, inner sep=1.2pt, label=above:\small{$b_4$}] (b4) at (8,4) {};
    
    \node[fill=none, draw=black, circle, inner sep=1.2pt, label=below:$a_1$] (a1) at (2,0) {};
    \node[fill=none, draw=black, circle, inner sep=1.2pt, label=below:$a_2$] (a2) at (4,0) {};
    \node[fill=none, draw=black, circle, inner sep=1.2pt, label=below:$a_3$] (a3) at (6,0) {};
    \node[fill=none, draw=black, circle, inner sep=1.2pt, label=below:$a_4$] (a4) at (8,0) {};

    \draw[fill=orange, draw=none] (0,2) circle [radius=0.12cm]; 
    \draw[white, line width=1.2pt] (-0.1,2.1) -- (0.1,1.9); 
    \draw[white, line width=1.2pt] (-0.1,1.9) -- (0.1,2.1); 
    \node (oo) at (0,2) {};
    
    \draw[fill=orange, draw=none] (1.5,2) circle [radius=0.12cm]; 
    \fill[white] (1.5,2) circle (0.05cm); 
    \node (oi) at (1.5,2) {};
    \draw[fill=springgreen, draw=none] (3,2) circle [radius=0.12cm]; 
    \draw[white, line width=1.2pt] (2.9,2.1) -- (3.1,1.9); 
    \draw[white, line width=1.2pt] (2.9,1.9) -- (3.1,2.1); 
    \node (go) at (3,2) {};
    
    \draw[fill=springgreen, draw=none] (4.5,2) circle [radius=0.12cm]; 
    \fill[white] (4.5,2) circle (0.05cm); 
    
    \node (gi) at (4.5,2) {};
    
    \draw[fill=lavender, draw=none] (6,2) circle [radius=0.12cm]; 
    \draw[white, line width=1.2pt] (5.9,2.1) -- (6.1,1.9); 
    \draw[white, line width=1.2pt] (5.9,1.9) -- (6.1,2.1); 

    \node (lo) at (6,2) {};
    
    \draw[fill=lavender, draw=none] (7.5,2) circle [radius=0.12cm]; 
    \fill[white] (7.5,2) circle (0.05cm); 

    \node (li) at (7.5,2) {};
    
    \draw[fill=royalblue, draw=none] (9,2) circle [radius=0.12cm]; 
    \draw[white, line width=1.2pt] (8.9,2.1) -- (9.1,1.9); 
    \draw[white, line width=1.2pt] (8.9,1.9) -- (9.1,2.1); 

    \node (bo) at (9,2) {};
    
    \draw[fill=royalblue, draw=none] (10.5,2) circle [radius=0.12cm]; 
    \fill[white] (10.5,2) circle (0.05cm); 

    \node (bi) at (10.5,2) {};


    \draw[line width=0.0400cm, color=orange] (a1) -- (oo);
    \draw[line width=0.0400cm, color=springgreen] (a1) -- (go);
    \draw[line width=0.0400cm, color=royalblue] (a1) -- (bo);
    
    \draw[line width=0.0400cm, color=orange] (a2) -- (oi);
    \draw[line width=0.0400cm, color=orange] (a2) -- (oo);
    \draw[line width=0.0400cm, color=lavender] (a2) -- (lo);
    
    \draw[line width=0.0400cm, color=orange] (a3) -- (oi);
    \draw[line width=0.0400cm, color=springgreen] (a3) -- (gi);
    \draw[line width=0.0400cm, color=royalblue] (a3) -- (bi);

    \draw[line width=0.0400cm, color=royalblue] (a4) -- (bi);
    \draw[line width=0.0400cm, color=royalblue] (a4) -- (bo);
    \draw[line width=0.0400cm, color=lavender] (a4) -- (li);

   \draw[line width=0.0400cm, color=orange] (b1) -- (oo);
    \draw[line width=0.0400cm, color=springgreen] (b1) -- (gi);
    \draw[line width=0.0400cm, color=lavender] (b1) -- (li);
    
    \draw[line width=0.0400cm, color=orange] (b2) -- (oi);
    \draw[line width=0.0400cm, color=springgreen] (b2) -- (go);
    \draw[line width=0.0400cm, color=lavender] (b2) -- (li);
    
    \draw[line width=0.0400cm, color=lavender] (b3) -- (lo);
    \draw[line width=0.0400cm, color=springgreen] (b3) -- (gi);
    \draw[line width=0.0400cm, color=royalblue] (b3) -- (bo);

    \draw[line width=0.0400cm, color=royalblue] (b4) -- (bi);
    \draw[line width=0.0400cm, color=springgreen] (b4) -- (go);
    \draw[line width=0.0400cm, color=lavender] (b4) -- (lo);

    \end{tikzpicture}
        \caption{The middle link $\mathsf{L}_1$ corresponding to Figure \ref{fig:taiko_example_4_by_4}. A vertex with color $i$ color and a cross represents $(i, \text{out})$ and a vertex with color $i$ and a dot represents $(i, \text{in})$.
        }
        \label{fig:taiko_example_middle_link}
    \end{figure}
    
    \end{example*}

\subsection{Paths, Cycles, and Girth}

\begin{definition}[Graphs]\label{def:graphs}
A \textit{finite unoriented (simple) graph} is a pair \(\mathcal G=(V,E)\) with a finite vertex set \(V\)
and an edge set \(E\subseteq \bigl\{\{u,v\}\subseteq V : u\neq v\bigr\}\).
A \textit{finite oriented (directed) graph} is a pair \(\mathcal G=(V,E)\) with
\(E\subseteq \bigl\{(u,v)\in V\times V : u\neq v\bigr\}\).
The \textit{underlying undirected graph} \(U(\mathcal G)\) has the same vertex set and
\(\{u,v\}\in E\bigl(U(\mathcal G)\bigr)\) iff \((u,v)\in E\) or \((v,u)\in E\).
\end{definition}

\begin{definition}[Walks, paths, cycles]\label{def:walk-path-cycle}
Let \(\mathcal G=(V,E)\) be a graph (unoriented or directed).
A \textit{walk of length \(k\)} is a sequence \(v_0,\dots,v_k\) with
\(\{v_{i-1},v_i\}\in E\) in the unoriented case, and \((v_{i-1},v_i)\in E\) in the directed case, for all \(i\).
A \textit{path} is a walk with all vertices distinct.
A \textit{cycle} is a closed walk \(v_0,\dots,v_k\) with \(k\ge 3\), \(v_0=v_k\), and
\(v_0,v_1,\dots,v_{k-1}\) all distinct. 
\end{definition}

\begin{definition}[Girth and half-girth]\label{def:girth}
The \textit{girth} of an unoriented graph \(\mathcal G\), denoted \(\girth(\mathcal G)\),
is the minimum length of a cycle in \(\mathcal G\), with the convention \(\girth(\mathcal G)=\infty\) if no cycle exists.
For a directed graph, unless otherwise stated, we define \(\girth(\mathcal G):=\girth\bigl(U(\mathcal G)\bigr)\).
The \textit{half-girth} of \(\mathcal G\) is \(\hgirth(\mathcal G):=\girth(\mathcal G)/2\), with \(\infty/2:=\infty\).
\end{definition}
\begin{remark}[Girth ignores orientation on $\mathsf L_A,\mathsf L_B$]
Although the horizontal graphs $\mathsf L_A$ and $\mathsf L_B$ are oriented, all girth
calculations use their underlying undirected graphs:
\[
  \girth(\mathsf L_A):=\girth\big(U(\mathsf L_A)\big),\ 
  \girth(\mathsf L_B):=\girth\big(U(\mathsf L_B)\big),\ 
  \girth(\mathsf L_{AB}) = \min\{\girth(\mathsf L_A), \girth(\mathsf L_B)\}
\]
where $U(\cdot)$ denotes the underlying undirected graph.
\end{remark}

\subsection{Combinatorial Conditions on Product Substructures}

	Our goal is to find a product structure $\Pi = (A,B,P)$ that satisfies the conditions $\mathsf{T_1}$ - $\mathsf{T_4}$:
\begin{tconds}

    \item {\sf orientation}. There exists a orientation function $O$ as in Definition \eqref{def:orientation}. Thus we want $\Pi$ to be orientable (Definition \eqref{def:orientablepi}). 
    
    \smallskip
    In the following conditions we assume \(\Pi\) is orientable.
    
    \item  {\sf no-fold}. A \textit{fold} in a taiko is a pair of horizontal edges that are incident to the same vertex $v \in A \sqcup B$, have the same color, and the same direction at $v$, meaning that they are either both incoming towards \(v\) or both outgoing from \(v\). $\Pi$ is said to satisfy no-fold condition if there is no-fold at any vertex in a taiko. For instance, the taiko in Figure \ref{fig:fold-example} has no fold. Note that the fold of on the right shown in Figure \ref{fig:fold-example} cannot occur at any vertex due to the definition of $E_{AB}$.
    
    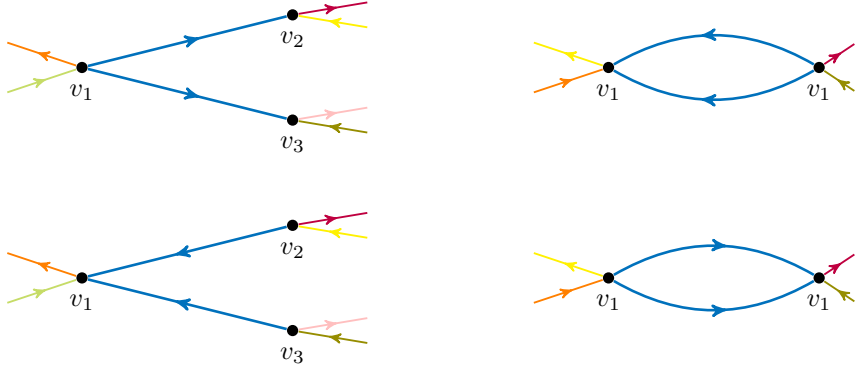
\begin{figure}[]
        \centering
                \tikzstyle arrowstyle=[scale=1]
        \tikzstyle arrowtipinmiddle=[postaction={decorate,decoration={markings,mark=at position .56 with {\arrow[arrowstyle]{stealth'}}}}]
        \begin{tikzpicture}[scale=0.7]
        \node[fill=black, circle, inner sep=1.5pt, label=below:\footnotesize{$v_1$}] (t1v1) at (0,5) {};
        \node[fill=black, circle, inner sep=1.5pt, label=below:\footnotesize{$v_2$}] (t1v2) at (4,6) {};
        \node[fill=black, circle, inner sep=1.5pt, label=below:\footnotesize{$v_3$}] (t1v3) at (4,4) {};
        \draw[fill=none, inner sep = 0.05cm] (-1.5,5.5)  node (w1) {};
        \draw[fill=none, inner sep = 0.05cm] (-1.5,4.5)  node (w2) {};

        \draw[fill=none, inner sep = 0.05cm] (5.5,6.25)  node (u1) {};
        \draw[fill=none, inner sep = 0.05cm] (5.5,5.75)  node (u2) {};
        
        \draw[fill=none, inner sep = 0.05cm] (5.5,4.25)  node (z1) {};
        \draw[fill=none, inner sep = 0.05cm] (5.5,3.75)  node (z2) {};

        \draw[arrowtipinmiddle, line width=0.02500cm, color=orange, bend left=0] (t1v1) to (w1);
        \draw[arrowtipinmiddle, line width=0.02500cm, color=springgreen, bend left=0] (w2) to (t1v1);

        \draw[arrowtipinmiddle, line width=0.03500cm, color=royalblue, bend left=0] (t1v1) to (t1v2);
        \draw[arrowtipinmiddle, line width=0.03500cm, color=royalblue, bend left=0] (t1v1) to (t1v3);

        \draw[arrowtipinmiddle, line width=0.02500cm, color=purple, bend left=0] (t1v2) to (u1);
        \draw[arrowtipinmiddle, line width=0.02500cm, color=yellow, bend left=0] (u2) to (t1v2);

        \draw[arrowtipinmiddle, line width=0.02500cm, color=pink, bend left=0] (t1v3) to (z1);
        \draw[arrowtipinmiddle, line width=0.02500cm, color=olive, bend left=0] (z2) to (t1v3);

        \node[fill=black, circle, inner sep=1.5pt, label=below:\footnotesize{$v_1$}] (t1v1) at (0,1) {};
        \node[fill=black, circle, inner sep=1.5pt, label=below:\footnotesize{$v_2$}] (t1v2) at (4,2) {};
        \node[fill=black, circle, inner sep=1.5pt, label=below:\footnotesize{$v_3$}] (t1v3) at (4,0) {};
        \draw[fill=none, inner sep = 0.05cm] (-1.5,1.5)  node (w1) {};
        \draw[fill=none, inner sep = 0.05cm] (-1.5,.5)  node (w2) {};

        \draw[fill=none, inner sep = 0.05cm] (5.5,2.25)  node (u1) {};
        \draw[fill=none, inner sep = 0.05cm] (5.5,1.75)  node (u2) {};
        
        \draw[fill=none, inner sep = 0.05cm] (5.5,0.25)  node (z1) {};
        \draw[fill=none, inner sep = 0.05cm] (5.5,-0.25)  node (z2) {};

        \draw[arrowtipinmiddle, line width=0.02500cm, color=orange, bend left=0] (t1v1) to (w1);
        \draw[arrowtipinmiddle, line width=0.02500cm, color=springgreen, bend left=0] (w2) to (t1v1);

        \draw[arrowtipinmiddle, line width=0.03500cm, color=royalblue, bend left=0] (t1v2) to (t1v1);
        \draw[arrowtipinmiddle, line width=0.03500cm, color=royalblue, bend left=0] (t1v3) to (t1v1);

        \draw[arrowtipinmiddle, line width=0.02500cm, color=purple, bend left=0] (t1v2) to (u1);
        \draw[arrowtipinmiddle, line width=0.02500cm, color=yellow, bend left=0] (u2) to (t1v2);

        \draw[arrowtipinmiddle, line width=0.02500cm, color=pink, bend left=0] (t1v3) to (z1);
        \draw[arrowtipinmiddle, line width=0.02500cm, color=olive, bend left=0] (z2) to (t1v3);
        
        \node[fill=black, circle, inner sep=1.5pt, label=below:\footnotesize{$v_1$}] (t2v1) at (10,1) {};
        \node[fill=black, circle, inner sep=1.5pt, label=below:\footnotesize{$v_1$}] (t2v2) at (14,1) {};

        \draw[fill=none, inner sep = 0.05cm] (8.5,1.5)  node (w1) {};
        \draw[fill=none, inner sep = 0.05cm] (8.5,.5)  node (w2) {};
        \draw[fill=none, inner sep = 0.05cm] (14.75,1.5)  node (x1) {};
        \draw[fill=none, inner sep = 0.05cm] (14.75,.5)  node (x2) {};

        \draw[arrowtipinmiddle, line width=0.03500cm, color=royalblue, bend left=30] (t2v1) to (t2v2);
        \draw[arrowtipinmiddle, line width=0.03500cm, color=royalblue, bend right=30] (t2v1) to (t2v2);
        \draw[arrowtipinmiddle, line width=0.02500cm, color=yellow, bend left=0] (t2v1) to (w1);
        \draw[arrowtipinmiddle, line width=0.02500cm, color=orange, bend left=0] (w2) to (t2v1);
        \draw[arrowtipinmiddle, line width=0.02500cm, color=purple, bend left=0] (t2v2) to (x1);
        \draw[arrowtipinmiddle, line width=0.02500cm, color=olive, bend left=0] (x2) to (t2v2);


        \node[fill=black, circle, inner sep=1.5pt, label=below:\footnotesize{$v_1$}] (t2v1) at (10,5) {};
        \node[fill=black, circle, inner sep=1.5pt, label=below:\footnotesize{$v_1$}] (t2v2) at (14,5) {};

        \draw[fill=none, inner sep = 0.05cm] (8.5,5.5)  node (w1) {};
        \draw[fill=none, inner sep = 0.05cm] (8.5,4.5)  node (w2) {};
        \draw[fill=none, inner sep = 0.05cm] (14.75,5.5)  node (x1) {};
        \draw[fill=none, inner sep = 0.05cm] (14.75,4.5)  node (x2) {};

        \draw[arrowtipinmiddle, line width=0.03500cm, color=royalblue, bend left=30] (t2v2) to (t2v1);
        \draw[arrowtipinmiddle, line width=0.03500cm, color=royalblue, bend right=30] (t2v2) to (t2v1);
        \draw[arrowtipinmiddle, line width=0.02500cm, color=yellow, bend left=0] (t2v1) to (w1);
        \draw[arrowtipinmiddle, line width=0.02500cm, color=orange, bend left=0] (w2) to (t2v1);
        \draw[arrowtipinmiddle, line width=0.02500cm, color=purple, bend left=0] (t2v2) to (x1);
        \draw[arrowtipinmiddle, line width=0.02500cm, color=olive, bend left=0] (x2) to (t2v2);
        \end{tikzpicture}
        \caption{Kinds of folds: folds shown on the right cannot occur if $\Pi$ is orientable.}
        \label{fig:fold-example}
    \end{figure}
    
    \item {\sf no-pattern}.  A pattern in a taiko is an unordered pair of colors of horizontal edges together with their orientations, that occur incident at a common vertex \(v\) in \(\mathsf{L}_{AB} = \mathsf{L}_A \cup \mathsf{L}_B\). That is, a pattern is a pair of the form \(\{(c_1, d_1), (c_2, d_2)\}\), where \(c_1\) and \(c_2\) are colors, and \(d_1, d_2 \in \{\text{in}, \text{out}\}\). The no-pattern condition says that a given product structure has no repeating patterns. That is, each pattern occurs at most once in \(\mathsf{L}_{AB}\).

    \item {\sf triple-girth}. 
    The condition $\mathsf{girth}(p, q)$ requires that $\girth(\mathsf{L}_{AB}) \geq p$ and $\text{half-girth}(\mathsf{L}_1) \geq q$. The $\mathsf{triple-girth}$ condition is satisfied if $\girth(p, q)$ equals one of the following pairs: $(p, q) = (6, 3)$, $(p, q) = (4, 4)$, or $(p, q) = (3, 6)$.

\end{tconds}

Note that repetition of pattern at at least two distinct vertices generated a cycle of girth $4$ in $\mathsf{L}_1$ and hence $\text{half-girth}(\mathsf{L}_1) \le 2$. Therefore, it is sufficient to consider conditions $\mathsf{T_1}, \mathsf{T_2}$ and $\mathsf{T_4}$ and omit $\mathsf{T_3}$.
 
\begin{theorem}[Theorem 27 in \cite{Mineyev2024}]\label{thm:27}
For the conjunctions:
\begin{enumerate}
    \item[(1)] {\sf orientation} and $\sf{girth}(6, 3)$,
    \item[(1')] {\sf orientation}, {\sf no-fold}, {\sf no-pattern}, and $\sf{girth}(\mathsf{L}_{AB}) \geq 6$,
    \item[(2)] {\sf orientation} and $\sf{girth}(6, 3)(4, 4)(3, 6)$,
    \item[(2')] {\sf orientation}, \sf no-fold, and $\sf{girth}(6, 3)(4, 4)(3, 6)$
\end{enumerate}
the following implications hold: 
\(
(1) \iff (1'),  (2) \iff (2'), (1) \implies (2) \implies (3).
\)

If a product structure $\Pi$ of size $(m, n)$ satisfies at least one of the conjunctions $(1)$, $(1')$, $(2)$, $(2')$, then $\Pi$ is non-degenerate, and both universal groups $G_\Pi$ and $\bar{G}_\Pi$ are torsion-free. 

In particular, if $m \geq 2$ and $n \geq 2$, then the associated elements $a_\Pi$ and $b_\Pi$ in $\mathbb{Z}_2 \bar{G}_\Pi$ provide a counterexample to the unit conjecture when $mn$ is odd, and a counterexample to the zero-divisor conjecture when $mn$ is even. If, in addition, the product structure admits a signature, then the associated elements $a_\Pi$ and $b_\Pi$ in $R \bar{G}_\Pi$ give such counterexamples over any ring $R$ with unity.
\end{theorem}

By Theorem~\ref{thm:27}, it is enough to find an orientable product structure \(\Pi\) satisfying \(\mathsf{T_2}\) and \(\mathsf{T_4}\).

\subsection{Isomorphisms}

Let $S_m$ and $S_n$ be the permutation groups on $m$ and $n$ letters. For $\sigma\in S_m$ and $\tau\in S_n$, define the action
\[
(\sigma,\tau)\cdot(a_i,b_j):=(a_{\sigma(i)},\,b_{\tau(j)}),
\]
and extend the action to edge sets $E\subseteq A\times B$ by
\[
(\sigma,\tau)\cdot E \ :=\ \{(\sigma,\tau)\cdot e : e\in E\}\subseteq A\times B,
\]
and to a subpartition $P=\{C_i\}_{i\in I}$ by
\[
(\sigma,\tau)\cdot P \ :=\ \{\,(\sigma,\tau)\cdot C_i : i\in I\,\}.
\]

\begin{definition}[Isomorphic subpartitions]\label{def:iso-subpartition}
Two subpartitions $P,P'\subseteq A\times B$ are  {isomorphic}, denoted \(P\cong P'\) if there exists
$(\sigma,\tau)\in S_m\times S_n$ with $P'=(\sigma,\tau)\cdot P$.
\end{definition}

Let $\Pi=(A,B,P)$ be an oriented product structure and let $E_{AB}(\Pi)$ be the fixed set of horizontal edges in $\mathsf L_A\sqcup\mathsf L_B$ (Definition~\ref{def:Hab}). Let $k_\Pi$ be the number of equivalence classes of $E_{AB}(\Pi)$, and let
\[
\kappa_\Pi:\ E_{AB}(\Pi)\longrightarrow \{1,\dots,k_\Pi\}
\]
be the associated color map (we use ``$i$-th color" and ``$i$-th class" interchangeably).

\begin{definition}[Isomorphism of taikos]\label{def:iso-taiko-strong}
Let \(\Pi=(A,B,P)\) and \(\Pi'=(A,B,P')\) be product substructures with taikos
\(\mathcal T(\Pi)\), \(\mathcal T(\Pi')\).
Let \(\kappa_\Pi:E_{AB}(\Pi)\to C_\Pi\) and \(\kappa_{\Pi'}:E_{AB}(\Pi')\to C_{\Pi'}\) be the
color maps onto finite color sets \(C_\Pi, C_{\Pi'}\).

We say \(\mathcal T(\Pi)\) and \(\mathcal T(\Pi')\) are  {isomorphic}, written
\(\mathcal T(\Pi)\cong \mathcal T(\Pi')\), if there exist \((\sigma,\tau)\in S_m\times S_n\),
a bijection of color sets \(\rho:C_\Pi\to C_{\Pi'}\), and a function
\(\varepsilon:C_\Pi\to\{\pm1\}\) such that, with
\[
f:A\sqcup B\to A\sqcup B,\qquad f(a_i):=a_{\sigma(i)},\ \ f(b_j):=b_{\tau(j)},
\]
the following hold:
\begin{enumerate}[label=(\roman*)]
  \item  {Vertical part:} \((\sigma,\tau)\cdot P=P'\).
  \item  {Horizontal part with possible flips:} for every horizontal oriented edge
  \(e=(u,v)\in E_{AB}(\Pi)\) of color \(c=\kappa_\Pi(e)\),
  set
  \[
    e' \ :=\
    \begin{cases}
      (f(u),f(v)), & \varepsilon(c)=+1,\\
      (f(v),f(u)), & \varepsilon(c)=-1.
    \end{cases}
  \]
  Then \(e'\in E_{AB}(\Pi')\) and \(\kappa_{\Pi'}(e')=\rho(c)\).
\end{enumerate}
\end{definition}

\begin{definition}[Isomorphism of middle links]\label{def:iso-middle-link}
Let \(\Pi\) and \(\Pi'\) be product substructures with taikos \(\mathcal T(\Pi)\), \(\mathcal T(\Pi')\)
and middle links \(\mathsf L_1(\Pi)\), \(\mathsf L_1(\Pi')\).
Let \(C_\Pi\) (resp.\ \(C_{\Pi'}\)) be the color set of horizontal classes for \(\Pi\) (resp.\ \(\Pi'\)),
and set the middle vertices
\[
M(\Pi):= C_{\Pi} \times \{\mathrm{in}, \mathrm{out}\},\qquad
M(\Pi'):= C_{\Pi'} \times \{\mathrm{in}, \mathrm{out}\}.
\]
We say \(\mathsf L_1(\Pi)\) and \(\mathsf L_1(\Pi')\) are  {isomorphic}, written
\(\mathsf L_1(\Pi)\cong \mathsf L_1(\Pi')\), if there exist \((\sigma,\tau)\in S_m\times S_n\),
a bijection of color sets \(\rho:C_\Pi\to C_{\Pi'}\), and a function
\(\varepsilon:C_\Pi\to\{\pm1\}\) such that:

\begin{enumerate}[label=(\roman*)]
  \item \text{Taikos are isomorphic with flips:}
  \(\mathcal T(\Pi)\cong \mathcal T(\Pi')\) via \((\sigma,\tau,\rho,\varepsilon)\) in the sense of
  Definition~\ref{def:iso-taiko-strong}; i.e., with vertex relabelling
  \(f(a_i)=a_{\sigma(i)}\), \(f(b_j)=b_{\tau(j)}\), color relabelling \(i\mapsto \rho(i)\),
  and for each color \(c\) all horizontal edges of color \(c\) are either kept or reversed
  according to \(\varepsilon(c)\in\{\pm1\}\).

  \item \text{Extend \(f\) to the middle vertices compatibly with the flips:}
  Let \(\iota\) be the swap on \(\{\mathrm{in},\mathrm{out}\}\) (i.e.\ \(\iota(\mathrm{in})=\mathrm{out}\),
  \(\iota(\mathrm{out})=\mathrm{in}\)).
  Define
  \[
  f\big((i,s)\big)\ :=\
  \begin{cases}
    (\rho(i),\,s), & \varepsilon(i)=+1,\\
    (\rho(i),\,\iota(s)), & \varepsilon(i)=-1,
  \end{cases}
  \qquad s\in\{\mathrm{in},\mathrm{out}\}.
  \]
  Then the bijection \(f: A\sqcup B\sqcup M(\Pi)\to A\sqcup B\sqcup M(\Pi')\) induces a graph isomorphism:
  \[
    \{u,v\}\in E\big(\mathsf L_1(\Pi)\big)\quad\Longleftrightarrow\quad
    \{f(u),f(v)\}\in E\big(\mathsf L_1(\Pi')\big).
  \]
\end{enumerate}
\end{definition}


\section{Our approach}
Throughout this section, let \( A := \{a_1, \ldots, a_m\} \) and \( B := \{b_1, \ldots, b_n\} \) denote two fixed finite sets. 

\subsection{Decreasing Conditions}
  The goal is to search for counterexamples of the conjectures by building a product structure recursively starting from the empty product substructure. This subsection outlines how the combinatorial conditions facilitate the recursive process.
  
    \begin{lemma}\label{lemma:girthdecreasing}
    Let $\mathcal{G}_1$ be a subgraph of $\mathcal{G}_2$, then $\mathrm{girth}(\mathcal{G}_1) \ge \mathrm{girth}(\mathcal{G}_2)$.
    \end{lemma}  
    \begin{proof} 
    If $\mathcal{G}_1$ has no cycles, we will say that $\mathrm{girth}(\mathcal{G}_1) = \infty$. The girth of $\mathcal{G}_2$ is finite (if it contains a cycle) and infinite otherwise. If $\mathcal{G}_1$ contains cycles, any nontrivial edge-loop in $\mathcal{G}_1$ is a nontrivial edge-loop in $\mathcal{G}_2$. Hence minimum of their lengths is larger for $\mathcal{G}_1$.    
    In both cases, we have shown that $\text{girth}(\mathcal{G}_1) \ge \text{girth}(\mathcal{G}_2)$.
    \end{proof}
\begin{corollary}\label{cor:decreasing_girth}
    Let $\Pi = (A,B, P)$ and $\Pi' = (A,B, P')$ be two product substructrues such that $P \subseteq P'$. Let $\mathsf{L}_{AB}$ and $\mathsf{L}_{AB}'$ be horizontal graphs associated to $\Pi$ and $\Pi'$, respectively, as defined in \eqref{eqn:defLAB}. Then, $\mathrm{girth}(\mathsf{L}_{AB}') \le \mathrm{girth}(\mathsf{L}_{AB})$. If $\mathsf{L}_1$ and $\mathsf{L}_1'$ are middle links for $\Pi$ and $\Pi'$, respectively, then $\hgirth(\mathsf{L}_1') \le \hgirth(\mathsf{L}_1)$.
    \end{corollary}
    \begin{proof}
        Since edges$(\mathsf{L}_A) \subseteq \mathrm{edges}(\mathsf{L}_A')$, $\mathrm{edges}(\mathsf{L}_B) \subseteq \mathrm{edges}(\mathsf{L}_B')$ and $\mathrm{edges}(\mathsf{L}_1) \subseteq \mathrm{edges}(\mathsf{L}_1')$, $\mathsf{L}_{AB}$ is a subgraph of $\mathsf{L}_{AB}'$ and $\mathsf{L}_1$ is a subgraph of $\mathsf{L}_1'$. The result follows from Lemma \ref{lemma:girthdecreasing}. 
    \end{proof}
        
    
    This lemma shows that as we create a larger subpartition of $A\times B$ by adding more cells, we decrease the girth of the horizontal graph and of the middle link.
    
   \begin{definition}[Decreasing condition]\label{def:decreasing_condition}
    A condition \(T\) is said to be \textit{decreasing} if the following holds:  Take arbitrary \(\Pi = (A, B, P)\) and \(\Pi' = (A, B, P')\) with  \(P \subseteq P'\). If \(T\) fails for \(\Pi\), then \(T\) must also fail for \(\Pi'\). Equivalently, \(T\) satisfies the \textit{hereditary property} in reverse: If \(T\) holds for \(\Pi'\), it must also hold for \(\Pi\).
    \end{definition}
    
    \begin{lemma}\label{lemma:hereditary}
    The conditions $\mathsf{T_1} - \mathsf{T_4}$ are decreasing conditions.
    \end{lemma}
    \begin{proof}
        Let $\Pi = (A, B, P)$ and $\Pi' = (A, B, P')$ be two partitions such that $P \subseteq P'$.

    $\mathsf{T_1}$, {\sf orientation}: Assume that $\Pi'$ satisfies orientation. This means there exists an orientation function $O$ (as defined in Definition \eqref{def:orientation}) that provides a consistent orientation for the partition $P'$. Since \(P \subseteq P'\), the orientation function \(O\), which orients the elements of \(P'\), is orients the elements of \(P\), as all cells in \(P\) are contained within \(P'\). 
    
    Thus, if $\mathsf{T_1}$ is satisfied by $\Pi'$, it must also be satisfied by $\Pi$, because the orientation function \(O\) remains valid for the smaller partition \(P\). Conversely, if $\mathsf{T_1}$ is not satisfied by $\Pi$, it implies that no orientation function exists for \(P\). Since \(P \subseteq P'\), the orientation function \(O\) cannot exist for \(P'\) either.
    

    $\mathsf{T_2}$, {\sf no-fold}: If there is no fold in $\Pi'$ then since $P \subseteq P'$, there should be no fold in $\Pi$.

    $\mathsf{T_3}$, {\sf no-pattern}: Assume that $\Pi'$ satisfies the no-pattern condition $\mathsf{T_3}$, meaning no repeating pattern occurs in the partition $P'$. Since $P \subseteq P'$, any pattern occurring in $P$ is also a pattern in $P'$. If there were a repeating pattern in $P$, then that same repeating pattern would also occur in $P'$, as all edges and vertices in $P$ are also present in $P'$. This would contradict the assumption that $\Pi'$ satisfies the no-pattern condition. Therefore, no repeating pattern can occur in $P$ if $\Pi'$ satisfies $\mathsf{T_3}$.
    
    Conversely, if $\mathsf{T_3}$ is not satisfied by $\Pi$, meaning a repeating pattern occurs in $P$, then this pattern must also be present in $P'$, as $P \subseteq P'$. Hence, $\mathsf{T_3}$ cannot be satisfied by $\Pi'$.

    $\mathsf{T_4}$, {\sf triple-girth}: This follows from Corollary \ref{cor:decreasing_girth}.
\end{proof}

Before adding a $2$-cell to a subpartition $P$, we relabel the cell using a process called Left-Alignment, which ensures the indices conform to a lexicographic order relative to the subpartition $P$. 

\subsection{Left alignment} 

Let $P := \{C_i\}_{i\in I}$ be a subpartition where each $C_i$ is a 2-cell of the form \[
C_i = \{(a_{i_1}, b_{j_1}), (a_{i_2}, b_{j_2})\}
\]
with $1\le i_1, i_2 \le m$ and $1\le j_1, j_2\le n$. Define the following indices:
\begin{equation*}
    i_P := \max_{C_i \in P} \{i_1, i_2 \},\ \ 
    j_P := \max_{C_i \in P}\{j_1, j_2\}
\end{equation*}

with the convention that $i_P = 0$ and $j_P = 0$ when $P = \emptyset$. $i_P$ and $j_P$ track the highest index of vertices that appear in $P$ from $A$ and $B$, respectively.

\begin{convention}[empty interval]\label{convention:empty_interval}
For integers \(u\le v\), write \(\{u,\dots,v\}\) for the set of integers from \(u\) to \(v\),
and set \(\{u,\dots,v\}=\varnothing\) when \(u>v\).
\end{convention}

\begin{definition}[Left alignment of a new cell]\label{def:left-alignment}
Let \(P=\{C_i\}_{i\in I}\) be a subpartition and let
\[
C=\{(a_{i_1},b_{j_1}),\ (a_{i_2},b_{j_2})\}
\]
be a \(2\)-cell disjoint from every cell of \(P\) such that \(i_1 \neq i_2\) and \(j_1 \neq j_2\).
Without loss of generality assume \(i_1<i_2\).
Left alignment is the following process for \(P\) and \(C\):
\begin{enumerate}[label=(\alph*)]
  \item Define the sets of new indices that exceed \(i_P\) and \(j_P\):
  \[
  N_A:=\{\,i\in\{i_1,i_2\}: i>i_P\,\},\qquad
  N_B:=\{\,j\in\{j_1,j_2\}: j>j_P\,\},
  \]
  and set \(r:=|N_A|\in\{0,1,2\}\), \(s:=|N_B|\in\{0,1,2\}\).

  \item Define the target slots using the empty-interval convention~\ref{convention:empty_interval} by
  \[
  T_A:=\{i_P+1,\dots,i_P+r\},\qquad
  T_B:=\{j_P+1,\dots,j_P+s\}.
  \]

  \item Let \(\phi_A:N_A\to T_A\) be the order-preserving bijection where
  \begin{itemize}
    \item if \(r=0\), the assignment is vacuous (equivalently, \(\phi_A\) is the unique bijection \(\varnothing\to\varnothing\));
    \item if \(r=1\), then \(\phi_A\) is the unique bijection \(N_A\to T_A\);
    \item if \(r=2\), then \(\phi_A(i_1)=i_P+1\) and \(\phi_A(i_2)=i_P+2\).
  \end{itemize}

  Define \(\phi_B:N_B\to T_B\) analogously, using the same order \((j_1,j_2)\)
  (that is, do not sort by value). In particular:
  \begin{itemize}
    \item if \(s=0\), the assignment is vacuous;
    \item if \(s=1\), then \(\phi_B\) is the unique bijection \(N_B\to T_B\);
    \item if \(s=2\), then \(\phi_B(j_1)=j_P+1\) and \(\phi_B(j_2)=j_P+2\).
  \end{itemize}

  \item Choose \(\sigma\in S_m\), \(\tau\in S_n\) satisfying
  \begin{enumerate}[label=(\roman*)]
    \item \(\sigma(i)=i\) for \(1\le i\le i_P\), and \(\sigma(i)=\phi_A(i)\) for \(i\in N_A\);
    \item \(\tau(j)=j\) for \(1\le j\le j_P\), and \(\tau(j)=\phi_B(j)\) for \(j\in N_B\).
  \end{enumerate}
\end{enumerate}
The \textit{left-aligned copy} of \(C\) relative to \(P\) is
\[
\mathcal A_P(C):=(\sigma,\tau)\cdot C.
\]
\end{definition}

By construction, the indices in \(\mathcal A_P(C) = \{(a_{\sigma(i_1)}, b_{\tau(j_1)}),(a_{\sigma(i_2)}, b_{\tau(j_2)})\}\) satisfy
\[
1\le {\sigma(i_1), \sigma(i_2)}\le i_P+2, \,\qquad
1\le \tau(j_1), \tau(j_2)\le j_P+2.
\]
This definition is independent of the choice of \((\sigma,\tau)\) subject to the above constraints.

Before proving the isomorphism statement for taikos, we illustrate the left-alignment procedure with a few examples.

\begin{example}\label{ex:left-align-trivial}
    If $i_1, i_2 \le i_P$ and $j_1, j_2 \le j_P$ then $\mathcal{A}(C) = C$ for any partition $P$ since $r = s = 0$, and $\sigma$ and $\tau$ fix first $i_P$ and $j_P$ elements, respectively . If $i_1, i_2 > i_P$ and $j_1, j_2 > j_P$, and we list the cell so that \(i_1<i_2\), then
      \(
        \mathcal A_P(C)=\{(a_{i_P+1},b_{j_P+1}),\ (a_{i_P+2},b_{j_P+2})\}.
      \)
\end{example}

\begin{example}[first iteration, empty subpartition]\label{ex:left-align-first-iteration}
  Let \(P\) be the empty subpartition (so \(i_P=j_P=0\)); see Figure~\ref{fig:alignining-left-process}.
  Let
  \[
    C=\{(a_{i_1},b_{j_1}),\ (a_{i_2},b_{j_2})\},\qquad
    1\le i_1<i_2\le m,\ \ 1\le j_1\neq j_2\le n .
  \]
  In Definition \ref{def:left-alignment}, $N_A = \{i_1, i_2\}, N_B=\{j_1, j_2\}, r=s=2$. To left-align \(C\) relative to \(P\), choose permutations \(\sigma\in S_m\) and \(\tau\in S_n\) with
  \[
    \sigma(1)=i_1,\ \ \sigma(2)=i_2,\qquad
    \tau(1)=j_1,\ \ \tau(2)=j_2 .
  \]
  Then
  \[
    \mathcal A_P(C)=(\sigma,\tau)\cdot C
    =\{(a_{\sigma^{-1}(i_1)},b_{\tau^{-1}(j_1)}),\ (a_{\sigma^{-1}(i_2)},b_{\tau^{-1}(j_2)})\}
    =\{(a_1,b_1),\ (a_2,b_2)\}.
  \]
  The full process is illustrated in Figure~\ref{fig:alignining-left-process}.
\end{example}

\begin{figure}[]
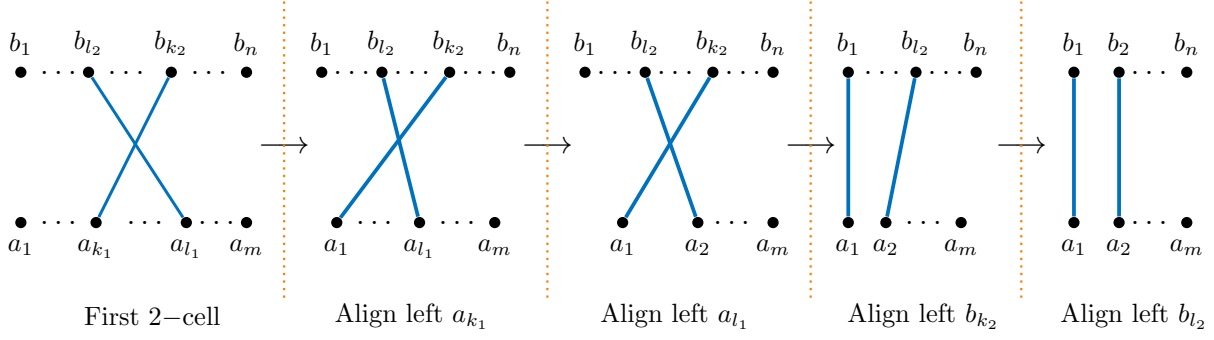

    \centering
    \include{flushleft_S1_process}
    \caption{$P = \emptyset, C = \{(a_{i_1}, b_{j_1}), (a_{i_2}, b_{j_2})\}, \ \mathcal{A}(C, \emptyset) = \{(a_1,b_1),(a_2,b_2)\}$}
    \label{fig:alignining-left-process}
\end{figure}
\begin{lemma}[Left alignment preserves isomorphism of subpartitions]\label{lemma:isomorphic_left_aligned_cells}
Let \(P=\{C_i\}_{i\in I}\) be a subpartition and
\(C\) a \(2\)-cell disjoint from \(P\).
Set \(P':=P\cup\{C\}\) and \( P'':=P\cup\{\mathcal A_P(C)\}\).
Then \(P'\) and \(P''\) are isomorphic subpartitions of \(\mathcal G(A,B)\).
\end{lemma}

\begin{proof}
Take \((\sigma,\tau)\) as in Definition~\ref{def:left-alignment}. By construction, \(\sigma\) fixes every
\(A\)-index \(\le i_P\) and \(\tau\) fixes every \(B\)-index \(\le j_P\); hence each cell of \(P\) is fixed setwise by \((\sigma,\tau)\).
Moreover, \((\sigma,\tau)\cdot C=\mathcal A_P(C)\).
Therefore \((\sigma,\tau)\cdot P'= P''\), so \(P'\) and \( P''\) are isomorphic.
\end{proof}

\begin{corollary}[Canonical left-aligned representative]\label{cor:unique-left-aligned-cell}
There exists a unique left-aligned representative of the cell $C$ up to isomorphism.
\end{corollary}

\begin{proof}
Existence follows from Lemma~\ref{lemma:isomorphic_left_aligned_cells}.
For uniqueness, the ordering $i_1 < i_2$ imposes a unique $\phi_A$ and hence, $\phi_B$ defined as in Definition \ref{def:left-alignment}. 
\end{proof}


\begin{lemma}\label{lemma:left-align-isomorphic-partitions}
With the notation of Lemma~\ref{lemma:isomorphic_left_aligned_cells},
let \(\Pi_i=(A,B,P)\), \(\Pi'=(A,B,P')\) and \(\Pi''=(A,B,P'')\).
 Then \(\mathcal T(\Pi')\) and \(\mathcal T(\Pi'')\) are isomorphic taikos.
Moreover, the induced permutation of color classes shows that both have the same number \(k\) of colors.
\end{lemma}
\begin{proof}
Let $E_{AB}, E_{AB}', E_{AB}''$ be horizontal edges of taikos $\mathcal{T}(\Pi), \mathcal{T}(\Pi'), \mathcal{T}(\Pi'')$, respectively. Let $k, k', k''$ be the number of colors (or equivalence classes) of horizontal edge set of $\mathcal{T}(\Pi)$, $\mathcal{T}(\Pi')$ and $\mathcal{T}({\Pi}'')$, respectively. Let $O$, $O'$ and $O''$ be the orientation functions with respect to $\Pi$, $\Pi'$ and $\Pi''$, respectively. Let $P = \{C_i\}_{i\in I}$, $P' = P \cup \{C\}$ where $C=\{(a_{i_1},b_{j_1}),\ (a_{i_2},b_{j_2})\}$ and ${P''} = P \cup \{\mathcal{A}_P(C)\}$. Note that $k \le \min\{k', k''\}$.

Let \((\sigma,\tau)\in S_m\times S_n\) be as in Lemma \ref{lemma:isomorphic_left_aligned_cells} so that \((\sigma,\tau)\cdot P=P'\). Define the vertex relabeling 
\[
f:A\sqcup B\longrightarrow A\sqcup B,\qquad f(a_i):=a_{\sigma(i)},\quad f(b_j):=b_{\tau(j)}.\]

\emph{Vertical edges.} If \((a_i,b_j)\) is a vertical edge of \(\mathcal T(\Pi)\) (i.e.\ lies in some \(C\in P\)), then \((a_{\sigma(i)},b_{\tau(j)})\) is a vertical edge of \(\mathcal T(\Pi')\) because \((\sigma,\tau)\cdot P'=P''\). Thus \(f\) identifies the vertical parts.

\emph{Undirected horizontal adjacencies.}  
Recall (by the definition of \(\bar E_A,\bar E_B\)) that
\(\{a_i,a_{i'}\}\in\bar E_A(\Pi')\) iff there exist \(b_j,b_{j'}\) with
\(\{(a_i,b_j),(a_{i'},b_{j'})\}\in P'\). Applying \((\sigma,\tau)\) gives
\(\{(a_{\sigma(i)},b_{\tau(j)}),(a_{\sigma(i')},b_{\tau(j')})\}\in P''\),
hence \(\{a_{\sigma(i)},a_{\sigma(i')}\}\in\bar E_A(\Pi'')\).
Thus \(f\) induces a bijection \(\bar E_A(\Pi')\cong \bar E_A(\Pi'')\).
The same argument on \(B\) shows \(f:\bar E_B(\Pi')\xrightarrow{\cong}\bar E_B(\Pi'')\),
so overall \(f:\bar E_{AB}(\Pi')\xrightarrow{\cong}\bar E_{AB}(\Pi'')\).

\emph{Oriented horizontal edges.}  
Assume \(\Pi'\) is orientable with orientation map \(O':\bar E_{AB}'\to E_{AB}'\) (Definition~\ref{def:orientation}). Note that either $k' = k$ or $k' = k+1$.

Case 1: If $k' > k$, we can choose $O'$ such that $O'\mid_{E_{AB}} = O$ ($O'$ when restricted to edge set of $\Pi$, coincides with $O$), $O'(\{a_{i_1}, a_{i_2}\}) = (a_{i_1}, a_{i_2})$ and $O'(\{b_{j_1}, b_{j_2}\})=(b_{j_1}, b_{j_2})$.

Case 2: If $k' = k$, then either $\{a_{i_1}, a_{i_2}\} \in \bar E_{AB}$ or $\{b_{j_1}, b_{j_2}\} \in \bar E_{AB}$. If $\{a_{i_1}, a_{i_2}\} \in \bar E_{AB}$. Then $O'(\{a_{i_1}, a_{i_2}\}) = O(\{a_{i_1}, a_{i_2}\})$ and this fixes an orientation on $\{b_{j_1}, b_{j_2}\}$. Otherwise $O'(\{b_{j_1}, b_{j_2}\}) = O(\{b_{j_1}, b_{j_2}\})$ and this fixes an orientation on $\{a_{i_1}, a_{i_2}\}$.

We define an orientation on \(\Pi''\) by {transport of structure}:
\[
O''\big(\{f(x),f(y)\}\big)\ :=\ f\big(O'(\{x,y\})\big)\qquad(\{x,y\}\in\bar E_{AB}').
\]
Because \(O'\) satisfies the \(2\)-cell compatibility on \(P'\) and \((\sigma,\tau)\cdot P'= P''\),
the map \(O''\) satisfies the same compatibility on \( P''\). Consequently,
\(
(a_i,a_{i'})\in E_A(\Pi)\ \Longrightarrow\ (a_{\sigma(i)},a_{\sigma(i')})\in E_A(\Pi''),
\)
and similarly on \(B\). Hence \(f\) is a directed graph isomorphism
\[
f:\ \mathsf L_A(\Pi)\sqcup\mathsf L_B(\Pi)\ \longrightarrow\ \mathsf L_A(\Pi'')\sqcup\mathsf L_B(\Pi'').
\]

\emph{Colors.}  
Let \(\kappa',\kappa'' :E_{AB}'\to\{1,\dots,\max\{k',k''\}\}\) be the color maps on horizontal edges (the classes of \(E_{AB}\)). Define \(\rho\in S_k\) by
\[
\rho\big(\kappa'(e)\big)\ :=\ \kappa''\big(f(e)\big)\qquad(e\in E_{AB}).
\]
This is well defined because two horizontal edges have the same color for $\Pi'$ iff their images under $f$ have the same color for $\Pi''$ (the equivalence relation is transported by
$(\sigma,\tau)$). Hence $\rho$ is a bijection, in particular $|C'|=|C''|$, so both taikos have the
same number $k$ of colors.

Combining the vertical identification, the horizontal directed-graph isomorphism, and the color
relabeling, the triple $(\sigma,\tau,\rho)$ realizes an isomorphism
$\mathcal T(\Pi')\cong \mathcal T(\Pi'')$ in the sense of Definition~\ref{def:iso-taiko-strong}.
\end{proof}

\begin{lemma}[Middle links are preserved under taiko isomorphism]\label{lem:L1-iso}
Let $\Pi$ and $\Pi'$ be product substructures. If
$\mathcal T(\Pi)\cong \mathcal T(\Pi')$ via $(\sigma,\tau,\rho,\varepsilon)$
in the sense of Definition~\ref{def:iso-taiko-strong}, then
\[
\mathsf L_1(\Pi)\ \cong\ \mathsf L_1(\Pi').
\]
\end{lemma}

\begin{proof}
Let $f:A\sqcup B\to A\sqcup B$ be the vertex relabelling
$f(a_i)=a_{\sigma(i)}$, $f(b_j)=b_{\tau(j)}$.
By taiko isomorphism, for every directed horizontal edge $e=(x,y)$ of
$\mathcal T(\Pi)$ with color $i=\kappa_\Pi(e)$, the image edge in $\mathcal T(\Pi')$ is
\[
e' \ =\
\begin{cases}
(f(x),\,f(y)), & \varepsilon(i)=+1,\\[2pt]
(f(y),\,f(x)), & \varepsilon(i)=-1,
\end{cases}
\qquad\text{with color }\ \kappa_{\Pi'}(e')=\rho(i).
\]

Extend $f$ to the middle vertices. Write $C_\Pi$ and $C_{\Pi'}$ for the color sets;
the middle vertices are $M(\Pi)=C_\Pi\times\{\mathrm{in},\mathrm{out}\}$ and
$M(\Pi')=C_{\Pi'}\times\{\mathrm{in},\mathrm{out}\}$.
Let $\iota$ be the swap on $\{\mathrm{in},\mathrm{out}\}$.
Define, for $s\in\{\mathrm{in},\mathrm{out}\}$,
\[
f\big((i,s)\big)\ :=\
\begin{cases}
(\rho(i),\,s), & \varepsilon(i)=+1,\\[2pt]
(\rho(i),\,\iota(s)), & \varepsilon(i)=-1.
\end{cases}
\]

Now use the combinatorial construction of $\mathsf L_1$.
From a horizontal edge $e=(x, y)$ of color $i$, $\mathsf L_1(\Pi)$ contains the
(undirected) edges $\{x,(i,\mathrm{out})\}$ and $\{y,(i,\mathrm{in})\}$.
Apply $f$:
\[
\{x,(i,\mathrm{out})\}\mapsto\{f(x),\,f(i,\mathrm{out})\},\qquad
\{y,(i,\mathrm{in})\}\mapsto\{f(y),\,f(i,\mathrm{in})\}.
\]

\noindent
$\bullet \ $  If $\varepsilon(i)=+1$, then $e'=(f(x), f(y))$ has color $\rho(i)$, and $\mathsf L_1(\Pi')$
adds $\{f(x),(\rho(i),\mathrm{out})\}$ and $\{f(y),(\rho(i),\mathrm{in})\}$,
which are exactly the edges added in \(\mathsf L_1(\Pi')\) from the colored edge \(f(e)\) because $f(i,\mathrm{out})=(\rho(i),\mathrm{out})$ and
$f(i,\mathrm{in})=(\rho(i),\mathrm{in})$.

\noindent
$\bullet \ $  If $\varepsilon(i)=-1$, then $e'=(f(y), f(x))$ has color $\rho(i)$, and $\mathsf L_1(\Pi')$
adds $\{f(y),(\rho(i),\mathrm{out})\}$ and $\{f(x),(\rho(i),\mathrm{in})\}$.
These match the images above because now
$f(i,\mathrm{out})=(\rho(i),\mathrm{in})$ and $f(i,\mathrm{in})=(\rho(i),\mathrm{out})$.

Thus $f$ is a bijection on vertices that preserves adjacency, hence an isomorphism
$\mathsf L_1(\Pi)\to \mathsf L_1(\Pi')$.
\end{proof}

\begin{corollary}\label{cor:L1-left-align}
With \(\Pi=(A,B,P')\) and \(\Pi''=(A,B,\bar P')\) as in Lemma~\ref{lemma:isomorphic_left_aligned_cells},
we have \(\mathsf L_1(\Pi)\cong \mathsf L_1(\Pi'')\).
\end{corollary}

\begin{proof}
By Lemma~\ref{lemma:left-align-isomorphic-partitions} the taikos are isomorphic; apply
Lemma~\ref{lem:L1-iso}.
\end{proof}


\begin{definition}[Valid subpartition]\label{def:valid-subpartition}
A subpartition $S\subseteq A\times B$ is \emph{valid} if it satisfies the disjoint vertex condition and its product substructure $\Pi=(A,B,S)$ satisfies the combinatorial conditions $\mathsf T_1$–$\mathsf T_4$.
\end{definition}

\begin{corollary}[Isomorphism invariance of validity]\label{cor:isomorphismvalidsubpartition}
If $\Pi=(A,B,P)$ and $\Pi'=(A,B,P')$ have isomorphic taikos, then
$P$ is valid iff $P'$ is valid.
\end{corollary}

\begin{proof}
Each condition $\mathsf T_1$–$\mathsf T_4$ is invariant under relabelling by $(\sigma,\tau)$ (and the induced permutation of colors), hence validity is preserved along orbits.
\end{proof}

\begin{remark}[{The effect of left alignment on duplicates}]
Left alignment does not by itself produce a unique representative of a whole subpartition up to isomorphism; however, it drastically reduces symmetry at each step by fixing the positions of the {new} indices of vertices in a $2-$cell that we want to add to a subpartition. In other words, any relabeling that preserves the current subpartition (up to isomorphism) must fix these newly used indices, so the automorphism group shrinks and candidates that differ only by such relabelings collapse to a single representative (fewer cases to check). 

For example, the number of $2$–cells in $A\times B$ is $  2\binom{m}{2}\binom{n}{2}$. Relative to the empty subpartition,
left alignment sends any such cell $C$ to \(
\mathcal{A}_\emptyset(C) = \{(a_1,b_1),(a_2,b_2)\} \) as shown in Figure \ref{fig:alignining-left-process}.
\end{remark}



\subsection{Search directed acyclic graph}\label{subsec:search-dag}

\begin{definition}[Search DAG]\label{def:search-dag}
A \textit{search directed acyclic graph (search DAG)} is a directed graph so that:
\begin{itemize}[leftmargin=*,label=\textbullet]
  \item \textit{Vertices.} A vertex is a valid left-aligned subpartition $P\subseteq A\times B$, i.e.\ a valid
  subpartition obtained from the root $\emptyset$ by iteratively adding left-aligned cells as in
  Definition~\ref{def:left-alignment}. Equivalently: there exists an ordering $P=\{C^1,\dots,C^t\}$
  such that $C^1=\mathcal A_{\emptyset}(C^1)$ and, inductively,
  $C^r=\mathcal A_{P^{r-1}}(C^r)$ for $P^{r-1}=\{C^1,\dots,C^{r-1}\}$, and each $P^{r}$ is valid.

  \item \textit{Root.} The root is $P_0:=\emptyset$.

  \item \textit{Edges.} There is a directed edge $(P, P')$ iff $P' = P \cup \{\mathcal A_P(C)\}$
  for some $2$-cell $C$ disjoint from $P$, and $P'$ is valid. (Children are always left-aligned
  relative to their parent.)
\end{itemize}
\end{definition}

\begin{proposition}[The search DAG is finite and acyclic]\label{prop:search-dag}
The directed graph of Definition~\ref{def:search-dag} is a finite directed acyclic graph
rooted at $\emptyset$.
\end{proposition}

\begin{proof}
Along any edge $(P, P')$ we have $|P'|=|P|+1$, so no directed cycle can occur.
Finiteness follows from finiteness of $A,B$ and the validity constraints.
Every vertex is obtained from $\emptyset$ by a finite sequence of one–cell extensions,
so $\emptyset$ is a root and the graph is connected from the root.
\end{proof}


\subsubsection*{Generating children}
Given a vertex $P$, we generate its outgoing edges as follows:
\begin{enumerate}
\item \textit{List candidates.} Enumerate all $2$–cells
$C=\{(a_{i_1},b_{j_1}),\,(a_{i_2},b_{j_2})\}$ that are disjoint from the cells of $P$.
\item \textit{Left-align.} For each candidate $C$, form its left-aligned image
$C^\dagger:=\mathcal A_P(C)$ (Definition~\ref{def:left-alignment}).
\item \textit{De-duplicate aligned cells.} Replace the multiset of aligned representatives by the set
\[
\mathcal C(P)\ :=\ \bigl\{\,C^\dagger=\mathcal A_P(C)\ \bigm|\ C\text{ disjoint from the cells of }P\,\bigr\},
\]
i.e.\ keep one representative for each distinct aligned $2$–cell $C^\dagger$.
\item \textit{Validity pruning.} For each $C^\dagger\in\mathcal C(P)$, if
$P':=P\cup\{C^\dagger\}$ is valid, add the edge $P\to P'$ to the search DAG.
\end{enumerate}

By Lemma~\ref{lemma:left-align-isomorphic-partitions}, replacing $C$ by $\mathcal A_P(C)$
preserves the isomorphism type; thus enumerating aligned representatives of cells and de-duplicating them
suffices to generate all valid extensions up to isomorphism.

\begin{remark}[Exploration tree for traversal]
While the intrinsic object is the search DAG, for traversal we maintain an \textit{exploration tree} --- a rooted directed spanning tree of the DAG (one parent per node, unique root-to-node path). Fix a deterministic total order on the aligned candidates $C^\dagger$ (e.g.\ lexicographic order on $(i_1<i_2;\,j_1,j_2)$). When a new vertex $P'$ is first discovered from $P$, record $P$ as the (unique) parent of $P'$ in the exploration tree. This parent choice depends on the traversal order and is not canonical, but it does not affect correctness; it simply provides a convenient backtracking structure.
\end{remark}

\subsection{The Algorithm}\label{subsec:algorithm}
We perform a depth–first search (DFS) on the search DAG to enumerate all valid product substructures $\Pi =(A,B,P)$ up to isomorphism. If $mn$ is even, each cell of $P$ consists of two edges, that is, each cell is a \textit{2-cell}. If $mn$ is odd, exactly one cell of $P$ is a $1-$cell and the rest are $2-$cells. A partition is built recursively, starting with an empty subpartition or a $1$-cell, and iterating by adding a $2$-cell until either a full partition is achieved or the current subpartition fails to satisfy the conditions $\mathsf{T}_1 - \mathsf{T}_4$.

\textsc{Input:} Fix $m,n$ and put $A=\{a_1,\dots,a_m\}$, $B=\{b_1,\dots,b_n\}$.

\textsc{DFS traversal (recursion):}

\begin{enumerate}
\item \textit{Initialization.} Start at the root $P=\emptyset$.
\item \textit{Expansion rule.} At a current vertex $P$, generate children according to the canonical
four–step procedure above and push each valid child $P'$ onto the DFS stack.
\item \textit{Advance/backtrack.} Pop the stack: if it is a new vertex, visit it and expand; otherwise backtrack.
\item \textit{Termination.} The search terminates when the stack is empty. Equivalently, one may stop
a branch when a full partition is reached (all edges paired if $mn$ is even, or exactly one $1$–cell
left if $mn$ is odd), or when no valid child exists.
\end{enumerate}

\begin{algorithm}
 \KwData{$(m,n)$ with $|A|=m$, $|B|=n$}
 \KwResult{All valid product substructures $\Pi=(A,B,P)$ up to isomorphism}
 initialize stack with $\emptyset$\;
 \While{stack not empty}{
   $P \leftarrow$ pop stack\;
    Output all possible children of $P$\;
     \ForEach{disjoint $2$–cell $C$}{
       $C^\dagger \leftarrow \mathcal A_P(C)$; apply canonical tie–break on the $B$–matching\;
       \If{$P' := P\cup\{C^\dagger\}$ is valid}{
         push $P'$ onto stack\;
       }
     }
 }
 \caption{Depth–first enumeration on the search DAG}\label{alg:dfs-search-dag}
\end{algorithm}


\begin{theorem}
    The Algorithm \ref{alg:dfs-search-dag} generates all product substructures 
    $\Pi=(A,B,P)$, where $P$ is a subpartition of $A\times B$, that satisfy $\mathsf{T}_1$–$\mathsf{T}_4$, 
    \textit{up to isomorphism}; that is, for every such $\Pi$ there exists an output 
    $\widehat\Pi$ of the algorithm with $\widehat\Pi \cong \Pi$. 
    (The algorithm need not produce a unique representative of each isomorphism class; 
    left alignment reduces but does not eliminate duplicates.)
\end{theorem}
\begin{proof}

    The proof is by induction on the number of $2$-cells in a subpartition. 
    
    The base case is when the number of 2-cells is one. Up to isomorphism, there is only one subpartition $P = \{C_1\}$ where $C_1 = \{(a_1,b_1), (a_2,b_2\}$ as seen in Example \ref{ex:left-align-first-iteration}. 
    
    Assume that we have obtained all subpartitions of cardinality $k-1$ where $k-1 < \left\lfloor \frac{mn}{2} \right\rfloor$. Let $P'$ be any subpartition of cardinality $k$ such that $P' = \{C_1, \ldots, C_k\}$ where each $C_k$ is a 2-cell. By Lemma \ref{lemma:hereditary}, if $P'$ satisfies the conditions $\mathsf{T}_1 - \mathsf{T}_4$, then they are satisfied by any subpartition $P$ such that $P \subseteq P'$. Let $P$ be such that $|P| = |P'|-1$ and $P \subsetneq P'$. Then by induction hypothesis, $P$ is generated by the algorithm. Without loss of generality, assume that $P'\setminus P = \{C_k\}$. Let $\mathcal{A}_P(C_k)$ be the left aligned cell corresponding to the cell $C_k$ and the subpartition $P$.
    
    By algorithm, in the expansion rule for $P$, all left-aligned cells up to isomorphism are generated. Thus $\mathcal{A}_P(C_k)$ is one the children of $P$ to be considered. Let $\Pi = (A,B,P \cup \{C_k\} )$ and $\Pi' = (A,B,P \cup \{\mathcal{A}(C_k)\})$. By Lemma \ref{lemma:left-align-isomorphic-partitions}, $\Pi$ and $\Pi'$ are isomorphic, and $\mathcal{T}(\Pi)$ and $\mathcal{T}(\Pi')$ are isomorphic. Thus, by Corollary \ref{cor:isomorphismvalidsubpartition}, $P \cup \{C_k\})$ is a valid subpartition if and only if $P \cup \{\mathcal{A}(C_k)\})$ is a valid subpartition. Thus, since $P' = P \cup \{C_k\}$ is valid, the algorithm will generate the subpartition $P \cup \{\mathcal{A}(C_k)\}$ since it is also valid. Thus $\Pi'$ is generated by algorithm which is isomorphic to $\Pi = (A,B,P')$.
\end{proof}

\section{Main  Results}

In this section, we prove Theorem \ref{thm:main} and Theorem \ref{thm:noexampletypemn}. Section 4.1 is devoted to the proof of Theorem \ref{thm:main} and Section 4.2 is devoted to the proof of Theorem \ref{thm:noexampletypemn}. Throughout, let \( A := \{a_1, \ldots, a_m\} \) and \( B := \{b_1, \ldots, b_n\} \) be two finite sets. Consider a partition \( P \) of \( A \times B \) that defines the product structure \( \Pi = (A, B, P) \). If \( \Pi \) satisfies the conditions \( \mathsf{T}_1 \) through \( \mathsf{T}_4 \), then \( \Pi \) is non-degenerate and provides a counterexample to the zero-divisor conjecture when \( mn \) is even, and to the unit conjecture when \( mn \) is odd, as established by Theorem \ref{thm:27}. 

This paper examines whether the sufficient conditions (1) and (2) in Theorem \ref{thm:27} can be met for various values of \( m \) and \( n \), and identifies the pairs \( (m, n) \) for which at least one condition fails.

\begin{definition}[Counterexample of type $(m,n)$]
A counterexample to either the zero-divisor conjecture or the unit conjecture is said to be of type \( (m, n) \) if there exists a product structure \( \Pi = (A, B, P) \) where \( |A| = m \), \( |B| = n \), and \( P \) is a partition of \( A \times B \) that satisfies conditions \( \mathsf{T}_1 \) through \( \mathsf{T}_4 \), with \( \Pi \) being either even or odd. 
\end{definition}


\subsection{All valid subpartitions up to cardinality 3}
In this subsection, we find all valid subpartitions and hence, the associated product substructures, taikos and middle-links for cardinality up to $3$. To efficiently track and present the combinatorial cases we utilize and track two lists and two indices for a subpartition $P$ throughout this section across different lemmas:
\begin{enumerate}
    \item \textit{List of vertices utilized}, denoted $\mathcal{V}_A$ and $\mathcal{V}_B$: This list tracks the vertices from sets \( A \) and \( B \) that have already been used in the current subpartition $P$, that is, if $\{(a,b),(a',b')\} \in P$, then $a,a' \in \mathcal{V}_A$ and $b,b'\in \mathcal{V}_B$. We may recall for left-alignment we tracked the largest index of used vertices from $A$ and $B$ for $P$.
    \item  \textit{List of edges available for $P$}, $\mathcal{E}$: This list holds the edges that are still available to be added to future subpartitions. It serves as a pool of edges from which the algorithm selects pairs to form new 2-cells during each iteration.
    \item $i_a^P$ and $j_b^P$: These variables track the indices of the highest vertex from $A$ and $B$, respectively, used for $P$.
\end{enumerate}

    
    

These lists were also utilized for the algorithm \cite{garg-github} to implement left-alignment, lexicographical ordering and various combinatorial conditions. 

\textbf{Notation}: $P_0$ represents the empty subpartition which also served as the root node for tree search Algorithm \ref{alg:dfs-search-dag}. $P_{ij}$ represents $j^{\text{th}}$ child of the subpartition $P_i$. Apriori there is no ordering of children of any subpartition but we can endow an order by using lexicographic order (or any order) on the edges. In our paper, we use lexicographic order on the edges, that is, $(a_{i_1}, b_{j_1})< (a_{i_2},b_{j_2})$ if either $i_1<i_2$ or $i_1=i_2$ and $j_1<j_2$. This endows an ordering on subpartitions.
\begin{lemma}\label{lemma:card1subpartitions}
    Up to isomorphism, there is exactly one product substructure \( \Pi = (A, B, P) \) such that \( |P| = 1 \), i.e., the subpartition \( P \) consists of a single element.
\end{lemma}

\begin{proof}
    The proof follows from Example \ref{ex:left-align-first-iteration} and the taiko corresponding to $\Pi = (A,B,P)$ is illustrated in Figure \ref{fig:first_iteration} where the red dotted edge depicts the next edge to be considered for the next cell. We obtain a subpartition $P_{01} = \{\{(a_1,b_1),(a_2,b_2)\}\}$ by applying the aligned-left condition to select the first $2-$cell. We will simply denote this subpartition as $P_1$.
    \begin{figure}
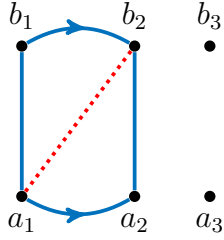

        \centering
        \include{first_iteration}
        \caption{$\mathcal{T}(\Pi)$ where $\Pi = (A,B,P)$ such that $P_1 = \{(a_1,b_1), (a_2,b_2)\}$.}
        \label{fig:first_iteration}
    \end{figure}

    After the first iteration, the list of vertices used for $A$ is $\mathcal{V}_A = (a_1, a_2)$ and for $B$, it is $\mathcal{V}_B = (b_1, b_2)$. Next we update the list of edges of to be considered and available for the cell for the next subpartition, $\mathcal{E}$, as follows: we consider all edges $(a_i, b_j)$ such that $a_i \in \mathcal{V}_A$ and $b_j \in \mathcal{V}_B$. We order these edges in lexicographic order. We will select the `smallest' edge according to this lexicographic order. Thus $\mathcal{E} = ((a_1,b_2),(a_2,b_1))$. $i_a = j_b = 2$.
\end{proof}

We will utilize $P_{1}, \va, \vb, \ecal$ from Lemma \ref{lemma:card1subpartitions} in the following lemma.
\begin{lemma}\label{lemma:card2subpartitions}
    Up to isomorphism, there are exactly three product substructures \( \Pi = (A, B, P) \) such that \( |P| = 2 \) and $P$ is a valid supartition, that is, $\Pi$ satisfies $\mathsf{T}_1-\mathsf{T}_4$. 
\end{lemma}
\begin{proof}
    Without loss of generality, by Lemma~\ref{lemma:card1subpartitions} we take $C_1=\{(a_1,b_1),(a_2,b_2)\}$. This means $P_{01} \subsetneq P$. By our lexicographic convention and by Lemma~\ref{lemma:card1subpartitions}, we require the smallest available edge $e_*=(a_1,b_2)\in\mathcal{E}$ to lie in $C_2$. 

    Let $(a_i,b_j)$ be the second edge of $C_2$. By the disjoint vertex condition, $i\neq 1$ and $j\neq 2$. Moreover, by left-alignment of $C_2$, $i,j \le 3$. Now, we consider the possibilities for the second edge in $C_2$. Let $i_a^{P_0}$ and $j_b^{P_0}$ be the highest indices of vertices used from $A$ and $B$, respectively, for $P_0$. Then $i_a^{P_0} = 2$ and $j_b^{P_0} = 2$.  

    All possible choices of $(a_i,b_j)$ are $(a_2, b_1), (a_2, b_3), (a_3, b_1)$ and $(a_3, b_3)$. By Example \ref{ex:2x2_not_possible}, $(a_2,b_1)$ yields an unorientable product structure and hence, failure of $\mathsf{T}_1$. Up to permutations, there are three possibilities for $C_2$ as shown in Figure \ref{fig:flush-left-second-cell}. Note that the half-girth in each case is infinite.
    
    \begin{figure}[]
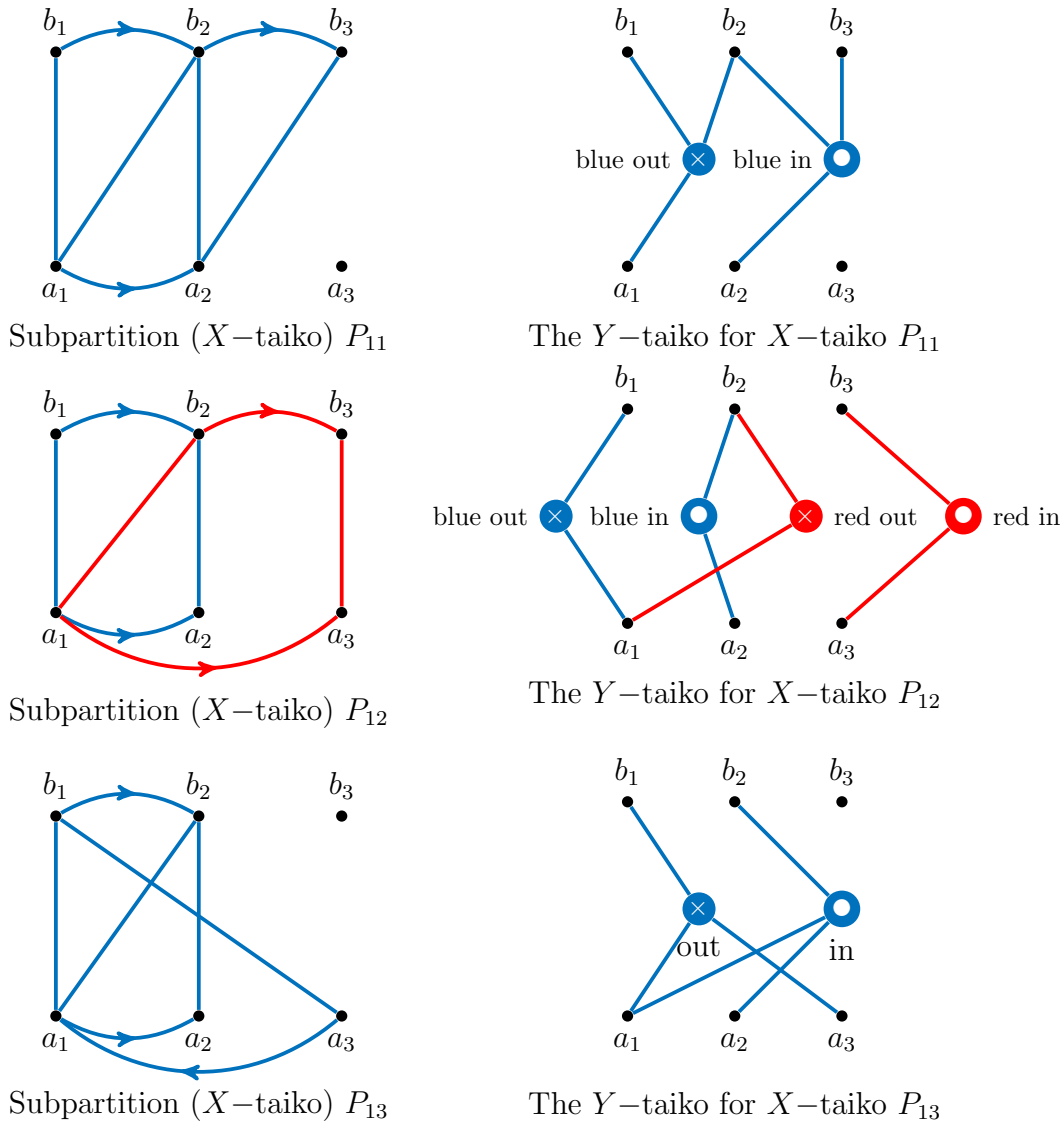

    \centering
    \include{flushleft_S1_children}
    \caption{$\mathcal{T}(\Pi)$ ($X-$taiko) for valid children of $P_{1}$ and their corresponding $\mathsf{L}_1$ graphs.}
    \label{fig:flush-left-second-cell}
    \end{figure}

    We update the lists $\mathcal{V}_A, \mathcal{V}_B, \mathcal{E}, i_a$ and $j_b$ as follows:
\begin{enumerate}
    \item For Subpartition $P_{11} = \{C_1, C_2\}$ where $C_1 = \{(a_1,b_1),(a_2,b_2)\}$ and $C_2 = \{(a_1,b_2),(a_2,b_3)\}$:
    \begin{itemize}
        \item $\mathcal{V}_A = (a_1, a_2)$, $\mathcal{V}_B = (b_1, b_2, b_3)$.
        \item The last vertex of highest index used from $A$ is $a_2$ and it is tracked by$i_a^{P_{11}} = 2$, and from $B$, it is $b_3$ tracked by $j_b^{P_{11}} = 3$.
        \item From $6$ possible edges between $\mathcal{V}_A$ and $\mathcal{V}_B$, we utilized $4$ for $C_1, C_2$. Therefore, $$\mathcal{E} = ((a_2,b_1), (a_1,b_3)).$$
    \end{itemize}
    \item For Subpartition $P_{12} = \{C_1, C_2\}$ where $C_1 = \{(a_1,b_1),(a_2,b_2)\}$ and $C_2 = \{(a_1,b_2),(a_3,b_3)\}$:
    \begin{itemize}
        \item $\mathcal{V}_A = (a_1, a_2,a_3)$, $\mathcal{V}_B = (b_1, b_2, b_3)$
        \item The last vertex of highest index used from $A$, tracked by $i_a^{P_{12}} =3$, is $a_3$, and from $B$, tracked by $j_b^{P_{12}}= 3$, is $b_3$
        \item From $9$ possible edges between $\mathcal{V}_A$ and $\mathcal{V}_B$, we utilized $4$ for $C_1, C_2$. Hence, $$\mathcal{E} = ((a_2,b_1), (a_1,b_3),(a_2,b_3), (a_3,b_1), (a_3,b_2)). $$
    \end{itemize}
    \item For Subpartition $P_{13} = \{C_1, C_2\}$ where $C_1 = \{(a_1,b_1),(a_2,b_2)\}$ and $C_2 = \{(a_3,b_1),(a_2,b_3)\}$:
    \begin{itemize}
        \item $\mathcal{V}_A = (a_1, a_2, a_3)$, $\mathcal{V}_B = (b_1, b_2)$
        \item The last vertex of highest index used from $A$, tracked by $i_a^{P_{13}}=3$, is $a_3$ and from $B$, tracked by $j_b^{P_{13}}=2$, is $b_2$
        \item From $6$ possible edges between $\mathcal{V}_A$ and $\mathcal{V}_B$, we utilized $4$ for $C_1, C_2$. Hence, $$\mathcal{E} = ((a_2,b_1), (a_3,b_2)).$$
    \end{itemize}
\end{enumerate}
\end{proof}

\begin{lemma}\label{lemma:card3subpartitions}
    Up to isomorphism, there are exactly five product substructure \( \Pi = (A, B, P^{(3)}) \) such that \( |P^{(3)}| = 3 \) and $P$ satisfies $\mathsf{T}_1-\mathsf{T}_4$. 
\end{lemma}

\begin{proof}
    For each of the subpartitions from the second iteration, we proceed similarly by selecting the next smallest edge from $L_3$, that is, $(a_2,b_1)$ according to the lexicographic order. We attempt to form new 2-cells, ensuring that the disjoint vertex condition and the conditions $\mathsf{T_1} - \mathsf{T_4}$ are satisfied. As in the second iteration, We keep track of the last vertex used for each $A$ and $B$. We add the next vertex $a_i$ (and $b_j$) to the vertex with the highest index used previously if it already does not exist in the list.

For instance, for subpartition $P_{11}$, we pick the smallest edge from $L_3$, which is $(a_2,b_1)$. The vertices of highest index used for $P_{11}$ are $a_2$ and $b_3$. Thus we add $a_3$ to $\mathcal{V}_A$ and $b_4$ to $\mathcal{V}_B$. We add all new vertical edges due to these new vertices in the list $L_3$. We pair the first edge with another edge from $L_3$, like $(a_1,b_3)$, to form a new 2-cell $C_3$. The algorithm will check if this pair satisfies the disjoint vertex condition and the conditions $\mathsf{T_1} - \mathsf{T_4}$ and will prune any invalid configurations. 

Up to permutations, we obtain $5$ children of $P_{11}$ (also shown in Figure \ref{fig:case1children}):
\begin{enumerate}
    \item $P_{111} = P_{11} \cup\{ C_3\}$ where $C_3 = \{(a_2,b_1),(a_1,b_3)\}$. This leads to repetition of pattern of (blue in, blue out), which is equivalent to \textit{half-girth} $2$.
    \item $P_{112} = P_{11} \cup\{ C_3\}$ where $C_3 = \{(a_2,b_1),(a_3,b_2)\}$. This leads to a cycle of \textit{half-girth} $2$ given as $(a_2, \text{ blue out}, b_2, \text{ blue in}, a_2)$ as shown in Figure \ref{fig:case1children}.
    \item $P_{113} = P_{11} \cup\{ C_3\}$ where $C_3 = \{(a_2,b_1),(a_3,b_3)\}$. This leads to \textit{half-girth} $3$ as shown.
    \item $P_{114} = P_{11} \cup\{ C_3\}$ where $C_3 = \{(a_2,b_1),(a_1,b_4)\}$. This leads to a blue edge incident at $b_1$ which is not possible.
    \item $P_{115} = P_{11} \cup\{ C_3\}$ where $C_3 = \{(a_2,b_1),(a_3,b_4)\}$. This leads to \textit{half-girth} $3$. Also, note that this is only valid child of $P_{11}$.
\end{enumerate}

\begin{figure}[]
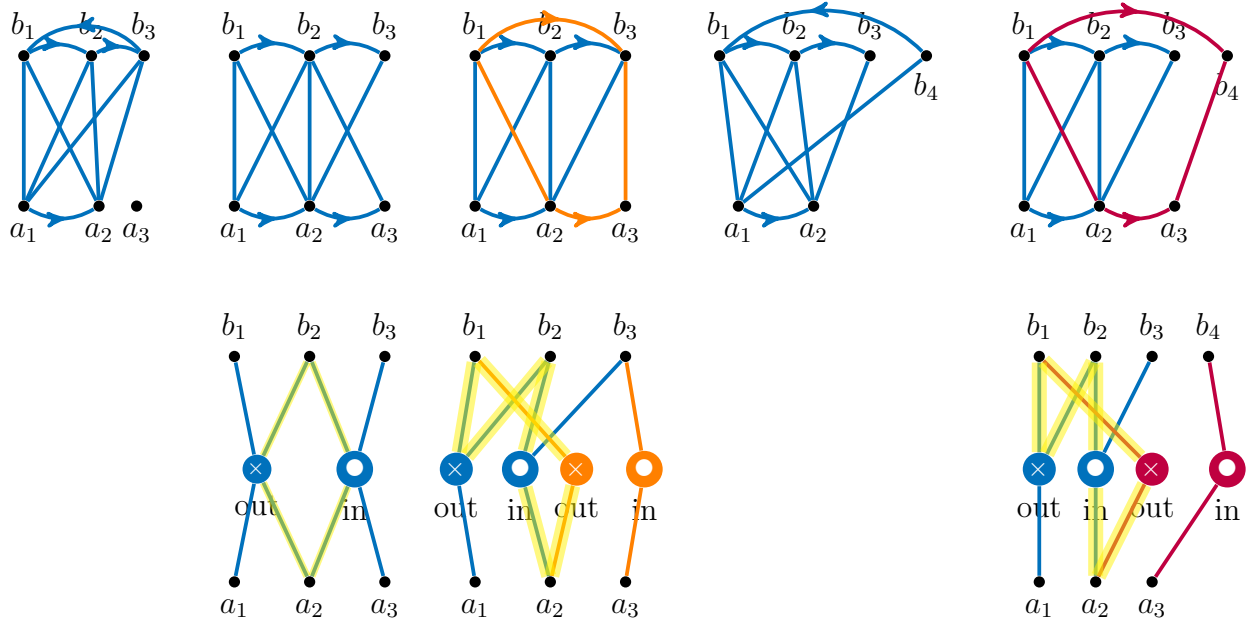

    \include{girth5c11children}
    \caption{All possible children (up to permutations) of $P_{11}$}
    \label{fig:case1children}
\end{figure}

Similarly, for $P_{12}$ we obtain $7$ children (up to permutations) (also shown in Figure \ref{fig:s012children1}):
\begin{enumerate}
    \item $P_{121} = P_{13} \cup\{ C_3\}$ where $C_3 = \{(a_2,b_1),(a_1,b_3)\}$. This subpartition has \textit{half-girth} $3$.
    \item $P_{122} = P_{13} \cup\{ C_3\}$ where $C_3 = \{(a_2,b_1),(a_3,b_2)\}$. This subpartition has \textit{half-girth} $3$.
    \item $P_{123} = P_{13} \cup\{ C_3\}$ where $C_3 = \{(a_2,b_1),(a_1,b_4)\}$. This subpartition has \textit{half-girth} $3$.
    \item $P_{124} = P_{13} \cup\{ C_3\}$ where $C_3 = \{(a_2,b_1),(a_3,b_4)\}$. This subpartition has \textit{half-girth}$(L_1)=4$ and $\text{girth}(L_{AB}) =\min\{3, \infty\} = 3$. Thus it fails triple girth condition.
    \item $P_{125} = P_{13} \cup\{ C_3\}$ where $C_3 = \{(a_2,b_1),(a_4,b_2)\}$. This subpartition has \textit{half-girth} $3$.
    \item $P_{126} = P_{13} \cup\{ C_3\}$ where $C_3 = \{(a_2,b_1),(a_4,b_3)\}$. This subpartition has \textit{half-girth}$(L_1)=4$ and $\text{girth}(L_{AB}) =\min\{\infty,3\} = 3$. Thus it fails triple girth condition.
    \item $P_{127} = P_{13} \cup\{ C_3\}$ where $C_3 = \{(a_2,b_1),(a_4,b_4)\}$. This subpartition has \textit{half-girth}$(L_1)=4$ and $\text{girth}(L_{AB}) = \infty$.
    
\end{enumerate}

\begin{figure}[]
    \centering
    \includegraphics[width=1\linewidth]{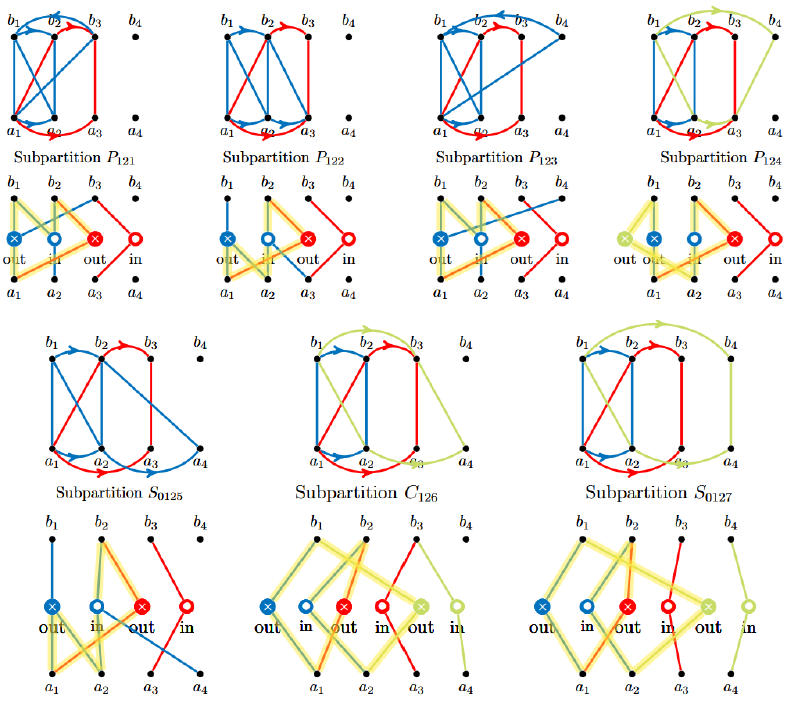}
    \caption{Product graph of valid children of $P_{12}$ along with their $Y-$taikos on the bottom.}
    \label{fig:s012children1}
\end{figure}


Furthermore, for the last child $P_{13}$ of $P_{1}$, we obtain the following $5$ children (also shown in Figure \ref{fig:S013children}):

\begin{enumerate}
    \item $P_{131} = P_{12} \cup\{ C_3\}$ where $C_3 = \{(a_2,b_1),(a_1,b_3)\}$. This subpartition has \textit{half-girth} $2$.
    \item $P_{132} = P_{12} \cup\{ C_3\}$ where $C_3 = \{(a_2,b_1),(a_3,b_2)\}$. This subpartition has \textit{half-girth} $2$.
    \item $P_{133} = P_{12} \cup\{ C_3\}$ where $C_3 = \{(a_2,b_1),(a_3,b_3)\}$. This subpartition has \textit{half-girth} $3$.
    \item $P_{134} = P_{12} \cup\{ C_3\}$ where $C_3 = \{(a_2,b_1),(a_4,b_2)\}$. This subpartition has \textit{half-girth} $2$.
    \item $P_{135} = P_{12} \cup\{ C_3\}$ where $C_3 = \{(a_2,b_1),(a_4,b_3)\}$. This subpartition has \textit{half-girth} $3$.
    
\end{enumerate}
\begin{figure}[]
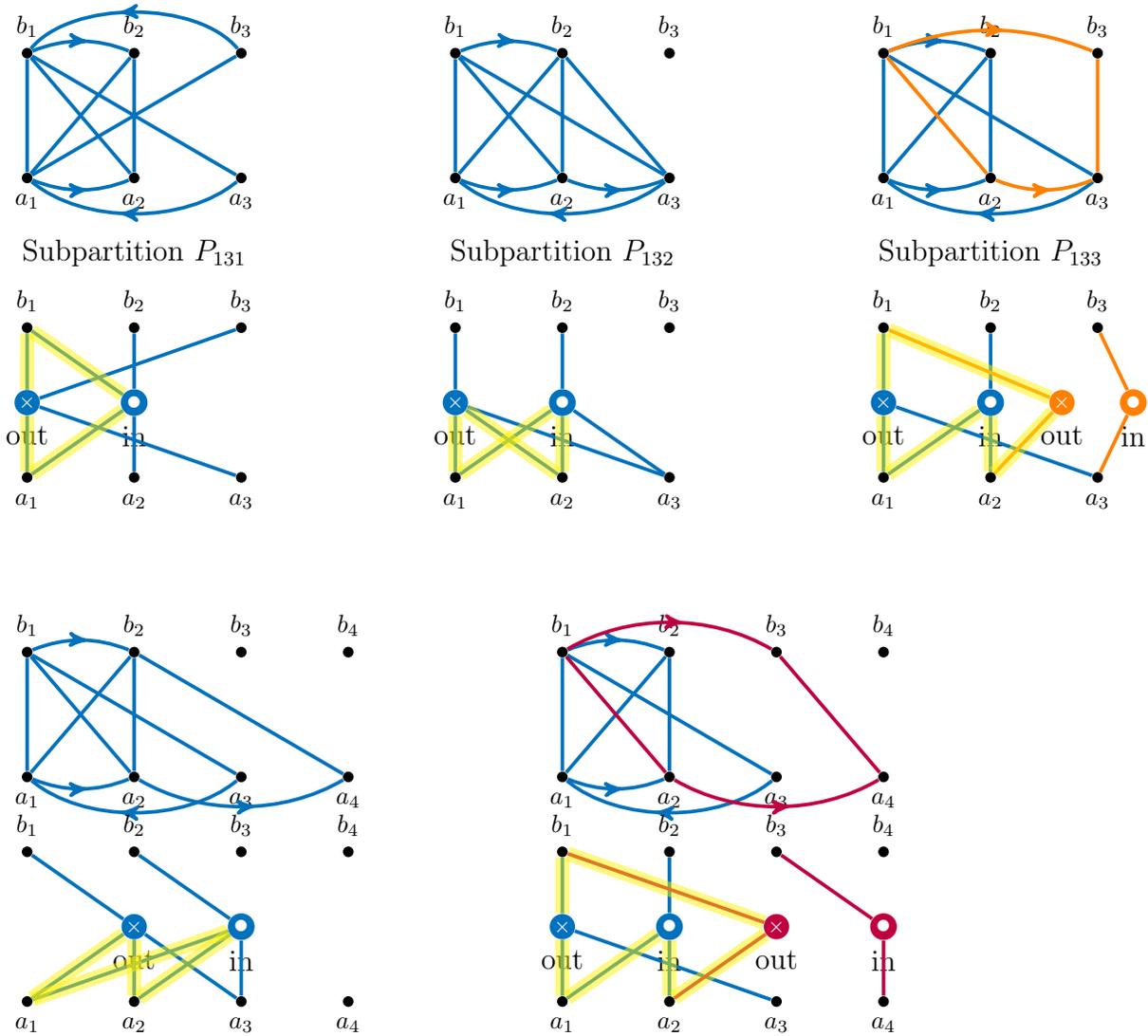

    \centering
    \include{girth5c13children}
    \caption{Possible children of $P_{13}$. $P_{131}, P_{132}, P_{133}$ and $P_{134}$ are not valid children}
    \label{fig:S013children}
\end{figure}

Thus, up to isomorphism (i.e., up to relabeling of $A$ and $B$), the only valid subpartitions of cardinality $3$ are
\[
P_{115},\; P_{123},\; P_{125},\; P_{127},\; \text{and } P_{135}.
\]

All other candidates either violate $\mathsf{T}_1$–$\mathsf{T}_4$ (e.g., unorientable or
failing the triple–girth condition) or are isomorphic to one of these five.
\end{proof}


\subsection{Proof of Theorem \ref{thm:main}}


If $|A| \le 2$ and $|B|\le 2$, there is no orientable product structure as shown in Example \ref{ex:2x2_not_possible}. Thus we can assume henceforth that $|P|\ge 3$. Choose any three $2$–cells of $P$ and let $P^{(3)}\subseteq P$ be the resulting subpartition. By
Lemma~\ref{lemma:card3subpartitions}, up to isomorphism there are exactly five possibilities for $P^{(3)}$, and in each case the associated middle–link graph satisfies $\hgirth\!\big(\mathsf L_1(P^{(3)})\big)\le 4$
(see Figures~\ref{fig:case1children}, \ref{fig:s012children1}, \ref{fig:S013children}).
Since $P^{(3)}\subseteq P$, Corollary~\ref{cor:decreasing_girth} yields
\[
\hgirth\!\big(\mathsf L_1(P)\big)\ \le\ \hgirth\!\big(\mathsf L_1(P^{(3)})\big)\ \le\ 4.
\]
\qed

\begin{remark}
    The triple-girth condition, $\mathsf{T}_4$, states that we need to look for an oriented product structure such that the girth pairs are either $(6,3)$, $(4,4)$ or $(3,6)$. By Theorem \ref{thm:main}, the conditions $(3,6)$ can never hold thus leading to pruning of one-third of the computation.
\end{remark}

\subsection{Counterexamples of type $(m,n)$}

Although it is well known in the literature that if $\alpha\beta = 0$ in a group ring $R[G]$ then $|\text{supp}(\alpha)|>2$ and $|\text{supp}(\beta)|>2$. For instance, see \cite{Schweitzer2013}. We show a rather obvious result to demonstrate the combinatorial technique to prove such statements for counterexamples of type $(m,n)$.

\begin{theorem}\label{thm:noexampletype2n}
    There are no counterexamples of type $(2,n)$ or $(n,2)$ where $n$ is a positive integer.
\end{theorem}
\begin{proof}
    The cases for $(2,n)$ and $(n,2)$ are identical by symmetry, therefore, we will only prove the result for $(2,n)$. When $n=2$, by Example \ref{ex:2x2_not_possible}, the condition $\mathsf{T}_1$ fails. 

    If $P$ is a partition of $A\times B$ then $|P| = n$. Let $C_1, \ldots, C_n$ denote $n$ 2-cells of $P$. By the left-alignment condition, we can assume that $C_1 = \{(a_1,b_1),(a_2,b_2)\} \in P$. Without loss of generality, we can assume that $(a_1,b_2) \in C_2$. By left-alignment, $(a_2,b_3) \in C_2$. Now consider the cell containing the edge $(a_2, b_1)$. We can assume it is $C_3$. Any edge $(a_1, b_j)$ will yield a horizontal edge $(b_j, b_1)$ in $\mathsf{L}_B$ which leads to girth $2$ in the associated middle-link graph, $\mathsf{L}_1$.
\end{proof}

\begin{theorem}\label{thm:noexampletype3n}
    There are no counterexamples of type $(3,n)$ or $(n,3)$ where $n$ is a positive integer.
\end{theorem}
\begin{proof}

    The proof is similar to the proof of Theorem \ref{thm:noexampletype2n}. Again, we will only work with the case when $A = \{a_1, a_2, a_3\}$ and $|B| = n \ge 3$. Note that  by Theorem \ref{thm:main}, we cannot have a cycle in $\mathsf{L}_A$ because otherwise $\girth(L_{AB}) = 3$ and $\girth(\mathsf{L}_1) \le 4$ and the subpartition would fail the triple girth condition. We have two cases:
    \begin{enumerate}
        \item Both edges in $\mathsf{L}_A$ are of the same color, that is, there is only one equivalence class of horizontal edges. The proof is identical to the proof of Theorem \ref{thm:noexampletype2n}.
        \item Both edges are of different color. For a subpartition of cardinality $3$, the only possibilities up to permutation are $P_{115}$ and $P_{123}$ by Lemma \ref{lemma:card3subpartitions}. For a child $P'$ of $P_{115}$, there will be horizontal edge $(b_3, b_j)$ for some $j \notin\{1,2,3\}$ such that it belongs to the same equivalence class as $(b_1, b_2)$ and $(b_2,b_3)$. Thus, $\girth(\mathsf{L}_1) = 2$ for $P'$.
        Similarly, consider the child $P'$ of $P_{123}$ so that $(a_2,b_4) \in C_k$ for some $C_k \in P'$. Then there is horizontal edge $(b_j,b_4)$ for some $j \notin\{1,2,3,4\}$. The edges $(b_1, b_2), (b_2, b_3), (b_1, b_4)$ and $(b_j,b_4)$ belong the same horizontal class thus leading to $\girth(\mathsf{L}_1) = 2$.
    \end{enumerate}
    \vspace{-0.5cm}
\end{proof}

\begin{theorem}\label{thm:4nnotpossible}
There are no counterexamples of type $(4,n)$ or $(n,4)$ where $n$ is a positive integer.
\end{theorem}

\begin{proof}
Let $n$ be a positive integer. We consider the case when $|A| = 4$ and $|B| = n$. The case when $|A| = n$ and $|B| = 4$ follows by symmetry. Since $|A| = 4$, the girth $\mathsf{L}_{A}$ is either less than or equal to $4$ or it is $\infty$. When $\girth(\mathsf{L}_{A})\le 4$, in the \text{triple girth condition}, $\mathsf{T_4}$, the only possible pairs are $(4,4)$ and $(3,6)$, where:
\begin{itemize}
    \item the pair $(4,4)$ means $\text{girth}(L_{AB}) \geq 4$ and $\textit{half-girth}(L_1) \geq 4$, and
    \item the pair $(3,6)$ means $\text{girth}(L_{AB}) \geq 3$ and $\textit{half-girth}(L_1) \geq 6$.
\end{itemize}
By Theorem~\ref{thm:main}, the only valid pair for $\mathsf{T_4}$ is $(4,4)$. However, as discussed in Section 4.2, every child of $P_{11}$ and $P_{13}$ yields $\textit{half-girth}(L_1) \leq 3$. The only subpartition yielding $\textit{half-girth}(L_1) = 4$ is $P_{127}$.

The subpartition $P_{127} = \{C_1, C_2, C_3\}$, where:
\begin{align*}
    C_1 &= \{(a_1,b_1), (a_2,b_2)\}, \\
    C_2 &= \{(a_1,b_2), (a_3,b_3)\}, \\
    C_3 &= \{(a_2,b_1), (a_4,b_4)\}.
\end{align*}
Let $C_4$ be a cell such that one of its edges is $(a_1, b_3)$. There are 9 possible choices for the second edge of $C_4$, namely:
\[
(a_2, b_4), (a_3, b_1), (a_3, b_2), (a_3, b_4), (a_4, b_1), (a_4, b_2), (a_2, b_5), (a_3, b_5), (a_4, b_5).
\]
For the choices $(a_3, b_1), (a_4, b_1), (a_4, b_2), (a_4, b_5)$, the triple girth condition is not satisfied since $\textit{girth}(L_{AB}) = 3$. For the pair $(a_3, b_2)$, the subpartition fails orientability. The remaining valid subpartitions are as follows:
\begin{itemize}
    \item $P_{1273} = P_{127} \cup C_4$, where $C_4 = \{(a_1, b_3), (a_2, b_4)\}$. The subpartition $P_{1273}$ further has 9 possibilities to pair with the edge $(a_2, b_3)$ to form a new cell $ C_5$, which are given as:
    \[
    (a_3, b_1), (a_3, b_2), (a_1, b_4), (a_3, b_4), (a_4, b_1), (a_4, b_2), (a_1, b_5), (a_3, b_5), (a_4, b_5).
    \]
    However, for $C_5 = \{(a_2, b_3),(a_1, b_4)\}$, orientability fails, and for the rest of the possibilities, the triple girth condition fails.
    \item $P_{1274} = P_{127} \cup C_4$ where $C_4 =\{(a_1, b_3), (a_3, b_4)\}$. The subpartition $P_{1274}$ further has $8$ possibilities to pair with the edge $(a_2, b_3)$ to form a new cell $ C_5$, which are given as:
    $$(a_3, b_1), (a_3, b_2), (a_1, b_4), (a_4, b_1), (a_4, b_2), (a_1, b_5), (a_3, b_5), (a_4, b_5).$$ However, for $C_5 = \{(a_2, b_3),(a_1, b_4)\}$, orientability fails and for rest of the possibilities triple girth condition fails.
    
    \item $P_{1277} = P_{127} \cup C_4$ where $C_4 =\{(a_1, b_3), (a_2, b_5)\}$. 
    The subpartition $P_{1277}$ further has $12$ possibilities to pair with the edge $(a_2, b_3)$ to form a new cell $ C_5$, which are given as:
    \[
    (a_3, b_1), (a_3, b_2), (a_1, b_4), (a_3, b_4), (a_4, b_1), (a_4, b_2), (a_1, b_5), (a_3, b_5), (a_4, b_5), (a_1, b_6), (a_3, b_6), (a_4, b_6).\] 
    For the possibilities $(a_1, b_5)$ and $(a_4, b_5)$ orientability fails and for rest of the possibilities triple girth condition fails.
     \item $P_{1278} = P_{127} \cup C_4$ where $C_4 =\{(a_1, b_3), (a_3, b_5)\}$. The subpartition $P_{1277}$ further has $11$ possibilities to pair with the edge $(a_2, b_3)$ to form a new cell $ C_5$, which are given as: 
     \[(a_3, b_1),  (a_3, b_2) ,  (a_1, b_4) ,  (a_3, b_4) ,  (a_4, b_1) ,  (a_4, b_2) ,  (a_1, b_5) ,  (a_4, b_5) ,  (a_1, b_6) ,  (a_3, b_6) ,  (a_4, b_6).\] 
     However, for $C_5 = \{(a_2, b_3),(a_1, b_5)\}$, orientability fails and for rest of the possibilities triple girth condition fails.
    
\end{itemize}

The case when $\girth(\mathsf{L}_{A}) = \infty$ is handled identically by computing all possible children for $P_{115},\; P_{123},\; P_{125},\; P_{127},\; \text{and } P_{135}$. This also yields no valid product substructures. Thus, no valid product structure exists when $|A| = 4$ and $|B| = n$. The proof for $|A| = n$ and $|B| = 4$ follows similarly by symmetry.
\end{proof}

\begin{theorem}\label{thm5nnotpossible}
There are no counterexamples of type $(5,n)$ or $(n,5)$ where $n$ is a positive integer.
\end{theorem}
\begin{proof}
    Consider the case when $\girth(\mathsf{L}_A) \le 5 $. $P_{127} = \{C_1, C_2, C_3\}$ where $C_1 = \{(a_1,b_1),(a_2,b_2)\}$, $C_2 = \{(a_1,b_2),(a_3,b_3)\}$ and $C_3 = \{(a_2,b_1),(a_4,b_4)\}$. In addition to children discussed in Theorem \ref{thm:4nnotpossible}, there are four more possibilities $(a_5, b_1), (a_5, b_2), (a_5, b_4), (a_5,b_5)$. 

    \begin{itemize}
    \item $P_{1273} = P_{127} \cup C_4$, where $C_4 = \{(a_1, b_3), (a_2, b_4)\}$. The subpartition $P_{1273}$ further has possible $13$ subpartitions up to permutations as children, however, for $12$ children triple girth condition fails and the orientation condition fails for one child.
    
    \item $P_{1274} = P_{127} \cup C_4$ where $C_4 =\{(a_1, b_3), (a_3, b_4)\}$. The subpartition $P_{1273}$ further has possible $12$ subpartitions up to permutations as children, however, for $11$ children triple girth condition fails and the orientation condition fails for one child.
    
    \item $P_{1277} = P_{127} \cup C_4$ where $C_4 =\{(a_1, b_3), (a_2, b_5)\}$. 
    The subpartition $P_{1277}$ further has $17$ possibilities to pair with the edge $(a_2, b_3)$ to form a new cell $ C_5$. For $15$ possibilities the triple girth condition is not satisfied and for two subpartitions the orientation fails.
    
    \item $P_{1278} = P_{127} \cup C_4$ where $C_4 =\{(a_1, b_3), (a_3, b_5)\}$. The subpartition $P_{1277}$ further has $16$ possibilities to pair with the edge $(a_2, b_3)$ to form a new cell $ C_5$. None of the possibilities yield a valid subpartition. For $15$ children triple girth condition fails and the orientation condition fails for one child.

    \item $P_{127 \{11\}} = P_{127} \cup C_4$ where $C_4 = \{(a_1, b_3),(a_5, b_2)\}$. This subpartition has $12$ further children; however, all children fail the triple girth conditions.

    \item $P_{127 \{12\}} = P_{127} \cup C_4$ where $C_4 = \{(a_1, b_3),(a_5, b_4)\}$. This subpartition has $12$ further children; however, all children fail the triple girth conditions.

    \item $P_{127 \{13\}} = P_{127} \cup C_4$ where $C_4 = \{(a_1, b_3),(a_5, b_5)\}$. This subpartition has $16$ further children; however, all children fail the triple girth conditions.
    
\end{itemize}
The case when $\girth(\mathsf{L}_{A}) = \infty$ is handled identically by computing all possible children for $P_{115},\; P_{123},\; P_{125},\; P_{127},\; \text{and } P_{135}$. This also yields no valid product substructures. Thus, no valid product structure exists when $|A| = 5$ and $|B| = n$. The proof for $|A| = n$ and $|B| = 5$ follows similarly by symmetry.
\end{proof}
    





\section{Computational Results}

The Algorithm \ref{alg:dfs-search-dag} generates all product substructures of size $(m,n)$ up to isomorphism such that the product substructures satisfy conditions $\mathsf{T}_1$ through $\mathsf{T}_4$. The algorithm is similar in spirit to the canonical labeling algorithm  in \cite{Schweitzer2013}. However, instead of choosing canonical representations in each iteration, we choose canonical representations of the cells we construct by using the left-alignment condition. A reference implementation of Algorithm~\ref{alg:dfs-search-dag} is publicly available at \cite{garg-github} which follows the pseudocode verbatim; left–alignment and the lexicographic tie–breaking to select edges for a 2-cell are enforced exactly as in §3.2–§3.4. This algorithm is also open-source and modular: new constraints can be added and existing ones toggled or removed, making it easy to tailor the search to stronger or weaker condition sets and to target specific families of potential counterexamples. 

The search explores a tree whose root is the empty subpartition; the level $k$ nodes are exactly the valid subpartitions (and hence product substructures) of cardinality $k$, up to isomorphism. Each node at height $k$ represents a subpartition with $k$ cells, and we record the height of every node. For example, at height $1$ there is a unique child (Lemma~\ref{lemma:card1subpartitions}); at height $2$ there are exactly three children (Lemma~\ref{lemma:card2subpartitions}). The first three levels are depicted in Figure~\ref{fig:tree}.



\begin{figure}
    \centering
    \tikzstyle arrowstyle=[scale=1]
\tikzstyle arrowtipinmiddle=[postaction={decorate,decoration={markings,mark=at position .56 with {\arrow[arrowstyle]{stealth'}}}}]
\tikzstyle verticalarrowtipinmiddle=[postaction={decorate,decoration={markings,mark=at position .7 with {\arrow[arrowstyle]{stealth'}}}}]
\tikzstyle 2arrowtipinmiddle=[postaction={decorate,decoration={markings,mark=at position .53 with {\arrow[arrowstyle]{{stealth'}}},mark=at position .58 with {\arrow[arrowstyle]{{stealth'}}}}}]
\tikzstyle vertical2arrowtipinmiddle=[postaction={decorate,decoration={markings,mark=at position .59 with {\arrow[arrowstyle]{{stealth'}}},mark=at position .62 with {\arrow[arrowstyle]{{stealth'}}}}}]
\tikzstyle 3arrowtipinmiddle=[postaction={decorate,decoration={markings,mark=at position .47 with {\arrow[arrowstyle]{stealth'}},mark=at position .56 with {\arrow[arrowstyle]{stealth'}},mark=at position .65 with {\arrow[arrowstyle]{stealth'}}}}]
\tikzstyle vertical3arrowtipinmiddle=[postaction={decorate,decoration={markings,mark=at position .33 with {\arrow[arrowstyle]{stealth'}},mark=at position .36 with {\arrow[arrowstyle]{stealth'}},mark=at position .39 with {\arrow[arrowstyle]{stealth'}}}}]
\tikzstyle trianglearrowtipinmiddle=[postaction={decorate,decoration={markings,mark=at position .56 with {\arrow[arrowstyle]{Triangle[open]}}}}]
\tikzstyle verticatrianglelarrowtipinmiddle=[postaction={decorate,decoration={markings,mark=at position .7 with {\arrow[arrowstyle]{Triangle[open]}}}}]
 \tikzstyle triangle2arrowtipinmiddle=[postaction={decorate,decoration={markings,mark=at position .53 with {\arrow[arrowstyle]{{Triangle[open]}}},mark=at position .58 with {\arrow[arrowstyle]{{Triangle[open]}}}}}]
\tikzstyle verticaltriangle2arrowtipinmiddle=[postaction={decorate,decoration={markings,mark=at position .59 with {\arrow[arrowstyle]{{Triangle[open]}}},mark=at position .62 with {\arrow[arrowstyle]{{Triangle[open]}}}}}]
\tikzstyle triangle3arrowtipinmiddle=[postaction={decorate,decoration={markings,mark=at position .47 with {\arrow[arrowstyle]{Triangle[open]}},mark=at position .56 with {\arrow[arrowstyle]{Triangle[open]}},mark=at position .65 with {\arrow[arrowstyle]{Triangle[open]}}}}]
\tikzstyle verticaltriangle3arrowtipinmiddle=[postaction={decorate,decoration={markings,mark=at position .33 with {\arrow[arrowstyle]{Triangle[open]}},mark=at position .36 with {\arrow[arrowstyle]{Triangle[open]}},mark=at position .39 with {\arrow[arrowstyle]{Triangle[open]}}}}]
    
\begin{tikzpicture}

    \node[fill=royalblue, circle, inner sep=1.5pt, label=above:\small{$\emptyset$}] (root) at (0,6) {};
    
    \node[fill=royalblue, circle, inner sep=1.5pt, label=left:\small{$P_1$}] (p1) at (0,4) {};
    \node[fill=royalblue, circle, inner sep=1.5pt, label=left:\small{$P_{11}$}] (p11) at (-3,2) {};
    \node[fill=royalblue, circle, inner sep=1.5pt, label=left:\small{$P_{12}$}] (p12) at (0,2) {};
    \node[fill=royalblue, circle, inner sep=1.5pt, label=left:\small{$P_{13}$}] (p13) at (3,2) {};

    \node[fill=royalblue, circle, inner sep=1.5pt, label=below:\small{$P_{115}$}] (p115) at (-3,0) {};

    \node[fill=royalblue, circle, inner sep=1.5pt, label=below:\small{$P_{125}$}] (p125) at (0,0) {};
    \node[fill=royalblue, circle, inner sep=1.5pt, label=below:\small{$P_{135}$}] (p135) at (3,0) {};
    \node[fill=royalblue, circle, inner sep=1.5pt, label=below:\small{$P_{133}$}] (p133) at (1.5,0) {};
    \node[fill=royalblue, circle, inner sep=1.5pt, label=below:\small{$P_{137}$}] (p137) at (4.5,0) {};

    \draw[line width=0.025cm, color=magenta] (root) -- (p1);
    \draw[line width=0.025cm, color=magenta] (p1) -- (p11);
    
    \draw[line width=0.025cm, color=magenta] (p1) -- (p12);
    \draw[line width=0.025cm, color=magenta] (p1) -- (p13);

    \draw[line width=0.025cm, color=magenta] (p11) -- (p115);
    \draw[line width=0.025cm, color=magenta] (p12) -- (p125);
    
    \draw[line width=0.025cm, color=magenta] (p13) -- (p135);
    
    \draw[line width=0.025cm, color=magenta] (p13) -- (p133);
    
    \draw[line width=0.025cm, color=magenta] (p13) -- (p137);

    \node[] (ht0) at (6.5,6) {\textcolor{royalblue}{\footnotesize{height 0}}};
    \node[] (ht0) at (6.5,4) {\textcolor{royalblue}{\footnotesize{height 1}}};
    \node[] (ht0) at (6.5,2) {\textcolor{royalblue}{\footnotesize{height 2}}};
    \node[] (ht0) at (6.5,0) {\textcolor{royalblue}{\footnotesize{height 3}}};

    \end{tikzpicture}

    \caption{Tree for the first three iterations of the algorithm as seen from Lemma \ref{lemma:card1subpartitions}, Lemma  \ref{lemma:card2subpartitions}, and Lemma \ref{lemma:card3subpartitions}}
    \label{fig:tree}
\end{figure}
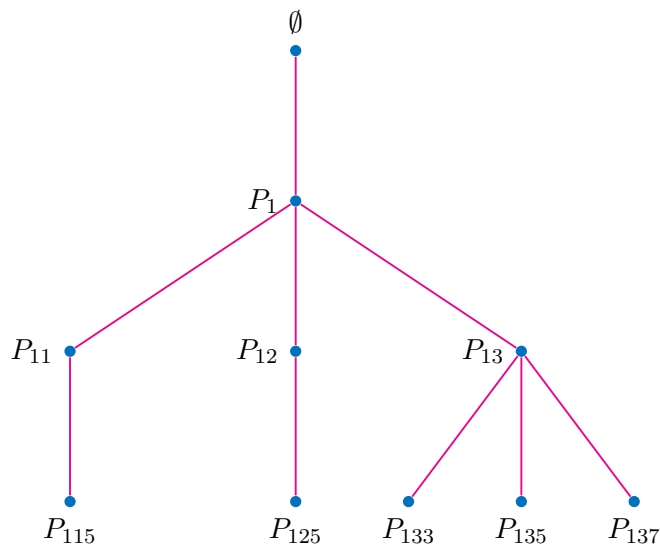


\begin{remark}[No example criterion]

Suppose the exploration tree attains maximum height $h$ (i.e., no node at height $h$
admits a valid child). Each new $2$–cell can introduce at most two previously unused
$A$–indices and at most two previously unused $B$–indices, so any branch of height $k$
uses at most $2k$ distinct indices on each side. Consequently:
\begin{enumerate}[label=(\roman*)]
\item if $2h<\min\{m,n\}$, then there is no product structure of any size $(m',n')$
with $m'\ge m$ and $n'\ge n$ satisfying $\mathsf T_1$–$\mathsf T_4$;
\item if $2h<n$, then there is no product structure of size $(m,n')$ with $n'\ge n$
satisfying $\mathsf T_1$–$\mathsf T_4$.
\end{enumerate}

\end{remark}


For $(m,n)$ in the range $1\le m,n\le 18$, no example criterion is not triggered, suggesting that the counterexamples might occur for higher sizes. These criteria can also determine when such a method of finding counterexamples might not be fruitful.  

\subsubsection*{Proof of Theorem \ref{thm:computer_assisted}} The execution of the algorithm has shown that for 
\begin{itemize}
    \item $1\le m\le 13$ and $1 \le n \le 13$,  or,
    \item  $m\in \{6,7\}$ and $1\le n\le 200$
\end{itemize}
there are no orientable product structures satisfying conditions $\mathsf{T}_1 - \mathsf{T}_4$.\qed

There exist many partial examples satisfying $\mathsf{T}_1$–$\mathsf{T}_3$.
For instance, Example~\ref{ex:taiko} has girth pair $(3,3)$.
Table~\ref{table:3-3} summarizes $(m,n)$ with $\mathsf{T}_1$–$\mathsf{T}_3$ and girth pair $(3,3)$; Table~\ref{table:4-4} summarizes the $(4,4)$ case.

\begin{table}
\centering
\begin{tabular}{l|l|l|l|l|l|l|l|l}
\multicolumn{1}{l}{} & \multicolumn{1}{l}{\small{4}} & \multicolumn{1}{l}{\small{5}} & \multicolumn{1}{l}{\small{6}} & \multicolumn{1}{l}{\small{7}} & \multicolumn{1}{l}{\small{8}} & \multicolumn{1}{l}{\small{9}} & \multicolumn{1}{l}{\small{10}} & $\tiny{\cdots}$  \\ 
\cline{2-9}
\small{4}                    &   \cellcolor{green!80} 
    \begin{tabular}{@{}c@{}}
    \tiny $~$ \\
    \end{tabular} & \cellcolor{green!80} 
    \begin{tabular}{@{}c@{}}
    \tiny $~$ \\
    \end{tabular}                       &            \cellcolor{green!80} 
    \begin{tabular}{@{}c@{}}
    \tiny $~$ \\
    \end{tabular}           & \cellcolor{green!80} 
    \begin{tabular}{@{}c@{}}
    \tiny $~$ \\
    \end{tabular}                      &            \cellcolor{green!80} 
    \begin{tabular}{@{}c@{}}
    \tiny $~$ \\
    \end{tabular}           &     \cellcolor{green!80} 
    \begin{tabular}{@{}c@{}}
    \tiny $~$ \\
    \end{tabular}                  &                        &      \\ 
\cline{2-9}
\small{5}                    &   \cellcolor{green!80} 
    \begin{tabular}{@{}c@{}}
    \tiny $~$ \\
    \end{tabular}                    & \cellcolor{green!80} 
    \begin{tabular}{@{}c@{}}
    \tiny $~$ \\
    \end{tabular}                      &        \cellcolor{green!80} 
    \begin{tabular}{@{}c@{}}
    \tiny $~$ \\
    \end{tabular}               &  \cellcolor{green!80} 
    \begin{tabular}{@{}c@{}}
    \tiny $~$ \\
    \end{tabular}                     &                       &                       &                        &      \\ 
\cline{2-9}
\small{6}                    &   \cellcolor{green!80} 
    \begin{tabular}{@{}c@{}}
    \tiny $~$ \\
    \end{tabular}                    & \cellcolor{green!80} 
    \begin{tabular}{@{}c@{}}
    \tiny $~$ \\
    \end{tabular}                       &      \cellcolor{green!80} 
    \begin{tabular}{@{}c@{}}
    \tiny $~$ \\
    \end{tabular}                 &                       &                       &                       &                        &      \\ 
\cline{2-9}
\small{7 }                   &  \cellcolor{green!80} 
    \begin{tabular}{@{}c@{}}
    \tiny $~$ \\
    \end{tabular}                     & \cellcolor{green!80} 
    \begin{tabular}{@{}c@{}}
    \tiny $~$ \\
    \end{tabular}                       &                       &                       &                       &                       &                        &      \\ 
\cline{2-9}
\small{8}                    &  \cellcolor{green!80} 
    \begin{tabular}{@{}c@{}}
    \tiny $~$ \\
    \end{tabular}                     &                       &                       &                       &                       &                       &                        &      \\ 
\cline{2-9}
\small{9}                    &   \cellcolor{green!80} 
    \begin{tabular}{@{}c@{}}
    \tiny $~$ \\
    \end{tabular}                    &                       &                       &                       &                       &                       &                        &      \\ 
\cline{2-9}
\small{10}                   &                       &                       &                       &                       &                       &                       &                        &      \\ 
\cline{2-9}
$\vdots$                  &                       &                       &                       &                       &                       &                       &                        &     
\end{tabular}
\caption{Green marks pairs $(m, n)$ where product structures of size $(m, n)$ satisfy $\mathsf{T}_1 - \mathsf{T}_3$ with girth pair $(3, 3)$.}
    \label{table:3-3}
\end{table}

\begin{table}[]
    \centering
    
    
\begin{tabular}{ c|c|c|c|c|c|c|c|c|c|c|c|c|c|c|c|c|c|c|c|c|c|c|c|c }
 & 4 & 5 & 6 & 7 & 8 & 9 & 10 & 11 & 12 & 13 & 14 & 15 & 16 & 17 & 18 & 19 & 20 & 21 & 22 & 23 & 24 & 25 & 26 & 27 \\

\\
 \hline6 &
 \cellcolor{red} \begin{tabular}{@{}c@{}}
 \tiny $~$ \\
\end{tabular}
 &
 \cellcolor{red} \begin{tabular}{@{}c@{}}
 \tiny $~$ \\
\end{tabular}
 &
 \cellcolor{red} \begin{tabular}{@{}c@{}}
 \tiny $~$ \\
\end{tabular}
 &
 \cellcolor{red} \begin{tabular}{@{}c@{}}
 \tiny $~$ \\
\end{tabular}
 &
 \cellcolor{red} \begin{tabular}{@{}c@{}}
 \tiny $~$ \\
\end{tabular}
 &
 \cellcolor{red} \begin{tabular}{@{}c@{}}
 \tiny $~$ \\
\end{tabular}
 &
 \cellcolor{red} \begin{tabular}{@{}c@{}}
 \tiny $~$ \\
\end{tabular}
 &
 \cellcolor{red} \begin{tabular}{@{}c@{}}
 \tiny $~$ \\
\end{tabular}
 &
 \cellcolor{red} \begin{tabular}{@{}c@{}}
 \tiny $~$ \\
\end{tabular}
 &
 \cellcolor{red} \begin{tabular}{@{}c@{}}
 \tiny $~$ \\
\end{tabular}
 &
 \cellcolor{red} \begin{tabular}{@{}c@{}}
 \tiny $~$ \\
\end{tabular}
 &
 \cellcolor{red} \begin{tabular}{@{}c@{}}
 \tiny $~$ \\
\end{tabular}
 &
 \cellcolor{red} \begin{tabular}{@{}c@{}}
 \tiny $~$ \\
\end{tabular}
 &
  &
  &
  &
  &
  &
  &
  &
  &
  &
  &
 
\\
 \hline7 &
 \cellcolor{red} \begin{tabular}{@{}c@{}}
 \tiny $~$ \\
\end{tabular}
 &
 \cellcolor{red} \begin{tabular}{@{}c@{}}
 \tiny $~$ \\
\end{tabular}
 &
 \cellcolor{red} \begin{tabular}{@{}c@{}}
 \tiny $~$ \\
\end{tabular}
 &
 \cellcolor{red} \begin{tabular}{@{}c@{}}
 \tiny $~$ \\
\end{tabular}
 &
 \cellcolor{red} \begin{tabular}{@{}c@{}}
 \tiny $~$ \\
\end{tabular}
 &
 \cellcolor{red} \begin{tabular}{@{}c@{}}
 \tiny $~$ \\
\end{tabular}
 &
 \cellcolor{red} \begin{tabular}{@{}c@{}}
 \tiny $~$ \\
\end{tabular}
 &
 \cellcolor{red} \begin{tabular}{@{}c@{}}
 \tiny $~$ \\
\end{tabular}
 &
 \cellcolor{red} \begin{tabular}{@{}c@{}}
 \tiny $~$ \\
\end{tabular}
 &
 \cellcolor{red} \begin{tabular}{@{}c@{}}
 \tiny $~$ \\
\end{tabular}
 &
 \cellcolor{red} \begin{tabular}{@{}c@{}}
 \tiny $~$ \\
\end{tabular}
 &
  &
  &
  &
  &
  &
  &
  &
  &
  &
  &
  &
  &
 
\\
 \hline8 &
 \cellcolor{red} \begin{tabular}{@{}c@{}}
 \tiny $~$ \\
\end{tabular}
 &
 \cellcolor{red} \begin{tabular}{@{}c@{}}
 \tiny $~$ \\
\end{tabular}
 &
 \cellcolor{red} \begin{tabular}{@{}c@{}}
 \tiny $~$ \\
\end{tabular}
 &
 \cellcolor{red} \begin{tabular}{@{}c@{}}
 \tiny $~$ \\
\end{tabular}
 &
 \cellcolor{red} \begin{tabular}{@{}c@{}}
 \tiny $~$ \\
\end{tabular}
 &
 \cellcolor{red} \begin{tabular}{@{}c@{}}
 \tiny $~$ \\
\end{tabular}
 &
 \cellcolor{red} \begin{tabular}{@{}c@{}}
 \tiny $~$ \\
\end{tabular}
 &
 \cellcolor{red} \begin{tabular}{@{}c@{}}
 \tiny $~$ \\
\end{tabular}
 &
 \cellcolor{red} \begin{tabular}{@{}c@{}}
 \tiny $~$ \\
\end{tabular}
 &
  &
  &
  &
  &
  &
  &
  &
  &
  &
  &
  &
  &
  &
  &
 
\\
 \hline9 &
 \cellcolor{red} \begin{tabular}{@{}c@{}}
 \tiny $~$ \\
\end{tabular}
 &
 \cellcolor{red} \begin{tabular}{@{}c@{}}
 \tiny $~$ \\
\end{tabular}
 &
 \cellcolor{red} \begin{tabular}{@{}c@{}}
 \tiny $~$ \\
\end{tabular}
 &
 \cellcolor{red} \begin{tabular}{@{}c@{}}
 \tiny $~$ \\
\end{tabular}
 &
 \cellcolor{red} \begin{tabular}{@{}c@{}}
 \tiny $~$ \\
\end{tabular}
 &
 \cellcolor{red} \begin{tabular}{@{}c@{}}
 \tiny $~$ \\
\end{tabular}
  &
 \cellcolor{red} \begin{tabular}{@{}c@{}}
 \tiny $~$ \\
\end{tabular}
 &
  &
  &
  &
  &
  &
  &
  &
  &
  &
  &
  &
  &
  &
  &
  &
  &
 
\\
 \hline10 &
 \cellcolor{red} \begin{tabular}{@{}c@{}}
 \tiny $~$ \\
\end{tabular}
 &
 \cellcolor{red} \begin{tabular}{@{}c@{}}
 \tiny $~$ \\
\end{tabular}
 &
 \cellcolor{red} \begin{tabular}{@{}c@{}}
 \tiny $~$ \\
\end{tabular}
 &
 \cellcolor{red} \begin{tabular}{@{}c@{}}
 \tiny $~$ \\
\end{tabular}
 &
 \cellcolor{red} \begin{tabular}{@{}c@{}}
 \tiny $~$ \\
\end{tabular}
 &
 \cellcolor{red} \begin{tabular}{@{}c@{}}
 \tiny $~$ \\
\end{tabular}
 &
  &
  &
  &
  &
  &
  &
  &
  &
  &
  &
  &
  &
  &
  &
  &
  &
  &
 
\\
 \hline11 &
 \cellcolor{red} \begin{tabular}{@{}c@{}}
 \tiny $~$ \\
\end{tabular}
 &
 \cellcolor{red} \begin{tabular}{@{}c@{}}
 \tiny $~$ \\
\end{tabular}
 &
 \cellcolor{red} \begin{tabular}{@{}c@{}}
 \tiny $~$ \\
\end{tabular}
 &
 \cellcolor{red} \begin{tabular}{@{}c@{}}
 \tiny $~$ \\
\end{tabular}
 &
 \cellcolor{red} \begin{tabular}{@{}c@{}}
 \tiny $~$ \\
\end{tabular}
 &
  &
  &
  &
  &
  &
  &
  &
  &
  &
  &
  &
  &
  &
  &
  &
  &
  &
  &
 
\\
 \hline12 &
 \cellcolor{red} \begin{tabular}{@{}c@{}}
 \tiny $~$ \\
\end{tabular}
 &
 \cellcolor{red} \begin{tabular}{@{}c@{}}
 \tiny $~$ \\
\end{tabular}
 &
 \cellcolor{red} \begin{tabular}{@{}c@{}}
 \tiny $~$ \\
\end{tabular}
 &
 \cellcolor{red} \begin{tabular}{@{}c@{}}
 \tiny $~$ \\
\end{tabular}
 &
 \cellcolor{red} \begin{tabular}{@{}c@{}}
 \tiny $~$ \\
\end{tabular}
 &
  &
  &
  &
  &
  &
  &
  &
  &
  &
  &
  &
  &
  &
  &
  &
  &
  &
  &
 
\\
 \hline13 &
 \cellcolor{red} \begin{tabular}{@{}c@{}}
 \tiny $~$ \\
\end{tabular}
 &
 \cellcolor{red} \begin{tabular}{@{}c@{}}
 \tiny $~$ \\
\end{tabular}
 &
 \cellcolor{red} \begin{tabular}{@{}c@{}}
 \tiny $~$ \\
\end{tabular}
 &
 \cellcolor{red} \begin{tabular}{@{}c@{}}
 \tiny $~$ \\
\end{tabular}
 &
  &
  &
  &
  &
  &
  &
  &
  &
  &
  &
  &
  &
  &
  &
  &
  &
  &
  &
  &
 
\\
 \hline14 &
 \cellcolor{red} \begin{tabular}{@{}c@{}}
 \tiny $~$ \\
\end{tabular}
 &
 \cellcolor{red} \begin{tabular}{@{}c@{}}
 \tiny $~$ \\
\end{tabular}
 &
 \cellcolor{red} \begin{tabular}{@{}c@{}}
 \tiny $~$ \\
\end{tabular}
 &
 \cellcolor{red} \begin{tabular}{@{}c@{}}
 \tiny $~$ \\
\end{tabular}
 &
  &
  &
  &
  &
  &
  &
  &
  &
  &
  &
  &
  &
  &
  &
  &
  &
  &
  &
  &
 
\\
 \hline15 &
 \cellcolor{red} \begin{tabular}{@{}c@{}}
 \tiny $~$ \\
\end{tabular}
 &
 \cellcolor{red} \begin{tabular}{@{}c@{}}
 \tiny $~$ \\
\end{tabular}
 &
 \cellcolor{red} \begin{tabular}{@{}c@{}}
 \tiny $~$ \\
\end{tabular}
 &
  &
  &
  &
  &
  &
  &
  &
  &
  &
  &
  &
  &
  &
  &
  &
  &
  &
  &
  &
  &
 
\\
 \hline16 &
 \cellcolor{red} \begin{tabular}{@{}c@{}}
 \tiny $~$ \\
\end{tabular}
 &
 \cellcolor{red} \begin{tabular}{@{}c@{}}
 \tiny $~$ \\
\end{tabular}
 &
 \cellcolor{red} \begin{tabular}{@{}c@{}}
 \tiny $~$ \\
\end{tabular}
 &
  &
  &
  &
  &
  &
  &
  &
  &
  &
  &
  &
  &
  &
  &
  &
  &
  &
  &
  &
  &

    \end{tabular}
    

    \caption{Red marks pairs $(m, n)$ where no product structures of size $(m, n)$ satisfy $\mathsf{T}_1 - \mathsf{T}_3$ with the girth pair $(4,4)$ and White represents squares represent the pairs $(m,n)$ for which the program has not been executed.}
    \label{table:4-4}
\end{table}


\subsection{Efficiency and refinement.}
The search remains highly effective due to two reinforcing features.

\noindent\textbf{(E1)} {Symmetry reduction via left–alignment and local canonicalisation.}

At every step, left–alignment quotients the index symmetries $S_m\times S_n$ by fixing the positions of newly introduced indices, so each expansion considers a single canonical representative of an isomorphism class of extensions.


Moreover, a fixed lexicographic tie–break on $B$–edges removes the two–matching ambiguity when both $B$–indices are new; de–duplication prevents equivalent matchings within the same parent from producing multiple children.

\noindent\textbf{(E2)} Early pruning by decreasing constraints.
The conditions $\mathsf T_1$–$\mathsf T_4$ are decreasing (Section~3): a violation at a node immediately excludes its entire descendant subtree. This focuses the search on structurally viable subpartitions and avoids exploring dead branches. 

Empirically, the combination of (E1) and (E2) reduces the effective branching factor by orders of magnitude relative to na\"ive enumeration of raw $2$–cells, making the method practical at substantial sizes.

\subsection{Complexity (baseline vs.\ left-aligned search).}
A naive baseline, ignoring symmetry and validity, treats each step as choosing two $A$–indices and two $B$–indices and one of the two matchings:
\[
N_{\mathrm{baseline}} \;=\; 2\,\binom{m}{2}\,\binom{n}{2}\quad\text{candidates per expansion.}
\]
If every branch reaches depth $t=\lfloor mn/2\rfloor$, this scales like $N_{\mathrm{baseline}}^{\,t}$. In the left-aligned search, after $k$ cells let $a_k\le \min\{2k,m\}$ and $b_k\le \min\{2k,n\}$ be the numbers of distinct indices actually used so far. Left alignment maps any truly new indices to the next slots $\{i_P{+}1,i_P{+}2\}$ and $\{j_P{+}1,j_P{+}2\}$ (collapsing permutations among unused indices), the fixed tie–break removes the two–matching ambiguity when both $B$–indices are new, and validity pruning discards many candidates immediately. A conservative upper bound on the aligned candidates at step $k$ is
\[
\widetilde N_k \;\le\; \binom{a_k}{2}\,\binom{b_k}{2}
\;\le\; \binom{\min\{2k,m\}}{2}\,\binom{\min\{2k,n\}}{2}.
\]
Thus along a branch of length $t$,
\[
\prod_{k=0}^{t-1}\widetilde N_k
\;\le\; \prod_{k=0}^{t-1}\binom{\min\{2k,m\}}{2}\,\binom{\min\{2k,n\}}{2},
\]
which is dramatically smaller than $\bigl(2\binom{m}{2}\binom{n}{2}\bigr)^{t}$ once $k\ll m,n$ and especially after pruning. In the common regime $2k<m,n$ this simplifies to
\[
\widetilde N_k \;\le\; \binom{2k}{2}\,\binom{2k}{2} \;=\; \Theta(k^4),
\]
so the per–level branching factor grows only polynomially in $k$ (before pruning), versus the constant but much larger baseline factor $\Theta(m^2n^2)$ at every level. 
\printbibliography
\vspace{-0.2cm}

\end{document}